\documentclass[10pt,a4pwide,reqno]{amsart}
\usepackage{amsmath}
\usepackage{amssymb}
\usepackage{amsfonts}
\usepackage{amsthm}
\usepackage{mathrsfs}
\usepackage{dsfont}
\usepackage{mathtools}
\usepackage{bm}
\usepackage{esint}
\usepackage{todonotes}

\usepackage[backref=page]{hyperref}
\usepackage{color}

\definecolor{darkgreen}{rgb}{0.5,0.25,0}
\definecolor{darkblue}{rgb}{0,0,1}
\definecolor{answerblue}{rgb}{0,0,0.75}

\hypersetup{colorlinks,breaklinks,
linkcolor=darkblue,urlcolor=darkblue,
anchorcolor=darkblue,citecolor=darkblue}

\newcommand*{\mailto}[1]{\href{mailto:#1}{\nolinkurl{#1}}}

\newcommand{\ep}{\varepsilon}
\newcommand{\pd}{\partial}

\newcommand{\Ex}{\mathbb{E}}

\renewcommand{\d}{\mathrm{d}}
\renewcommand{\ge}{\geqslant}
\renewcommand{\le}{\leqslant}

\newcommand{\doublehookrightarrow}{
    \lhook\joinrel\relbar\mspace{-12mu}\hookrightarrow
}

\newcommand{\one}[1]{\mathds{1}_{\{#1\}}}

\newcommand{\R}{\mathbb{R}}
\newcommand{\T}{\mathbb{T}}
\newcommand{\N}{\mathbb{N}}

\newcommand{\abs}[1]{\left | #1 \right |}
\newcommand{\norm}[1]{\left\| #1 \right\|}
\newcommand{\bk}[1]{ \left(  #1 \right)}

\makeatletter
\@namedef{subjclassname@2020}{%
  \textup{2020} Mathematics Subject Classification}
\makeatother

\theoremstyle{theorem}
\newtheorem{thm}{Theorem}[section]
\newtheorem{prop}[thm]{Proposition}
\newtheorem{lem}[thm]{Lemma}

\theoremstyle{definition}
\newtheorem{defin}[thm]{Definition}
\theoremstyle{remark}
\newtheorem{rem}[thm]{Remark}
\newtheorem{assume}[thm]{Assumption}

\numberwithin{equation}{section}

\allowdisplaybreaks

\title[viscous stochastic variational wave equation]
{The viscous variational wave equation with transport noise}

\author[P.H.C. Pang]{Peter H.C. Pang}
\address[Peter H.C. Pang]{Department of Mathematics\\
University of Oslo\\
NO-0316 Oslo\\ Norway}
\email{\mailto{ptr@math.uio.no}}



\date{\today}

\begin{document}

\begin{abstract}
This article considers the variational wave equation 
with viscosity and transport noise as a  
system of three coupled nonlinear stochastic partial differential equations. 
We prove pathwise global 
existence, uniqueness, and temporal continuity 
of solutions to this system in $L^2_x$. 
Martingale solutions are extracted from 
a two-level Galerkin approximation  
via the Skorokhod--Jakubowski theorem.  
We use the apparatus of Dudley maps to 
streamline this stochastic compactness method, 
bypassing the usual martingale identification 
argument. 
Pathwise uniqueness for the system 
is established through a renormalisation 
procedure that involves double commutator estimates 
and a delicate handling of noise and nonlinear terms. 
New model-specific commutator estimates 
are proven. 
\end{abstract}

\maketitle
\setcounter{secnumdepth}{2}

\setcounter{tocdepth}{1}
{\small \tableofcontents}

\section{Introduction}

\subsection{Background} 
We study the well-posedness and other 
solution properties of a viscous variational 
wave equation \eqref{eq:vvw} -- \eqref{eq:constitutive}
with transport type noise in this paper. 
Our problem is posed on $[0,T] \times \T$ 
and we understand $\T$ to be $[0,1]$ with 
periodic boundary.

The deterministic variational 
wave equation (rigorously studied in \cite{ZZ2001a} and 
numerous other works referenced below) is given by:
\begin{align}\label{eq:main_det}
\pd_{tt}^2 u - c(u) \,\pd_x \bk{c(u) \,\pd_x u} = 0.
\end{align}
It is natural to consider the 
equation using Riemann invariants 
\begin{align}\label{eq:riemann_inv}
R := \pd_t u + c(u) \,\pd_x u, \qquad 
S := \pd_t u - c(u) \,\pd_x u.
\end{align}
The variational wave equation 
\eqref{eq:main_det} can be formally 
transformed into the system 
\begin{equation}\label{eq:main_det2}
\begin{aligned}
\pd_t R - c(u)\, \pd_x R 
	& = \tilde{c}(u) (R^2 - S^2),\\
\pd_t S + c(u)\, \pd_x S 
	& = -\tilde{c}(u) (R^2 - S^2), \\
2 c(u) \pd_x u & = (R - S),
\end{aligned}
\end{equation}
where $\tilde{c} := c'/(4c)$ (note that 
some authors use the notation 
$\tilde{c}' := c'/(4c)$), with initial conditions
\begin{align*}
R^0 =  v^0 + c(u^0)\,\pd_x u^0, \qquad
 S^0 = v^0 - c(u^0) \,\pd_x u^0.
\end{align*}

There has been sustained interest in these 
partial differential equations (PDEs). 
The equation \eqref{eq:main_det} arises 
as the Euler--Lagrange equation 
of variational principle for the energy 
$\int_0^T \int_\T \abs{\pd_t u}^2+ \abs{c(u) \,\pd_x u}^2\,\d x\,\d t$. 
They are pertinent in systems 
such as wave maps from $4$-dimensional Minkowski space 
to $\mathbb{S}^2$, in geometric optics, 
in orientation waves of the director fields 
of nematic liquid crystals with Oseen--Franck potential 
energy; see, e.g., \cite[Section 2]{GHZ1997}, 
\cite{AH2007,HS1991}, and references there 
for a more complete discussion. 
Mathematically, like the related Camassa--Holm and 
Hunter--Saxton equations, the variational wave 
equation exhibits supercritical behaviour. 
Even for smooth initial data, solutions can 
exhibit wave-breaking, when $u$ remains 
bounded and continuous,
and $\pd_x u \sim {R - S}$ (or  $\pd_t u = R + S$, in the deterministic setting) 
becomes unbounded \cite{GHZ1996} 
(and \cite{DHW2020} for spherically symmetric solutions in higher dimensions). 
Continuation past wave-breaking 
is non-unique.
Notable among these continuations 
are conservative solutions, 
in which the associated energy 
is conserved a.e.~in time throughout the 
evolution, and dissipative solutions, 
for which the associated energy attenuates in time.

Dissipative weak solutions in $H^1_x$ to \eqref{eq:main_det} 
have been studied in the deterministic 
setting in \cite{ZZ2001a,ZZ2001b, 
	ZZ2003, ZZ2005a,ZZ2005b} using 
renormalisation techniques and Young measures. 
Using a variable transformation to deal with 
singularities arising from energy concentration, 
dissipative solutions have also been studied along characteristics 
in \cite{BH2016} under the positivity condition $c' > 0$. 
Related methods were used to establish 
well-posedness of conservative solutions 
in \cite{BCZ2015, BZZ2007,BZ2006} . 
A semigroup of conservative solutions 
was constructed in \cite{HR2011}, which 
also leveraged their methods to construct 
a scheme for computing conservative 
solutions numerically. 
A convergent finite difference scheme for 
dissipative solutions with $c' > 0$ and data 
$R^0, S^0 \le 0$ (sometimes called 
rarefactive solutions), 
was given in \cite{HKR2009}. Other finite 
difference schemes \cite{Web2016} and 
discontinuous Galerkin schemes \cite{AK2017a,AK2017b,LY2018} 
for \eqref{eq:main_det} and 2 (spatial) dimensional generalisations 
have been proposed and numerically verified. 
Weak travelling wave 
solutions were recently considered in \cite{GR2021}.

The inclusion of transport type noise in 
the study of variational wave-type equations 
is inspired by geometric-mechanical and 
physical considerations \cite{Fla2011,Holm:2015tc}. 
Transport type noises have proven mathematically rich in 
diverse contexts, of which the regularisation by 
noise phenomenon is a prominent 
example (see \cite{Flandoli:2021ab,FGP2010} and references contained there). 
By putting \eqref{eq:riemann_inv} in the 
stochastic wave equation with gradient noise
$$
\d v - c(u) \,\pd_x\bk{c(u) \,\pd_x u} = 2 \sigma \pd_x v \circ \d W, \qquad v = \,\pd_t u,
$$
we are led to consider the following stochastic variational 
wave equation:
\begin{equation}\label{eq:main_svw}
\begin{aligned}
\d R - c(u)\, \pd_x R \,\d t
	& = \tilde{c}(u) (R^2 - S^2)\,\d t 
	+ \sigma \,\pd_x \bk{R + S}\circ \d W,\\
\d S + c(u)\, \pd_x S \,\d t
	& = -\tilde{c}(u) (R^2 - S^2) \,\d t 
	+\sigma\, \pd_x \bk{R + S}\circ \d W , \\
2 c(u) \pd_x u & = (R - S), \qquad (t,x) \in [0,T] \times \T,
\end{aligned}
\end{equation}
where we recall that $\tilde{c} = c'(u)/(4c(u))$.

Subsequently throughout the paper, we maintain 
the following assumptions on the coefficients 
$c$ and $\sigma$:
\begin{assume}\label{sum:c_sigma}
There is a fixed constant $\kappa > 1$ for which 
 $c : \R \to [\kappa^{-1}, \kappa]$, and 
 $c \in C^1(\R)$. In particular, 
we assume that $\abs{c'} \le \kappa$ also. 
Moreover, we assume that $\sigma \in W^{2,\infty}(\T)$. 
\end{assume}
Our assumptions 
immediately imply that 
$\abs{\tilde{c}}:= c'/\bk{4c} \lesssim \kappa^2$, 
but crucially, we do not assume that $c'$ is positive.

Formal small amplitude or high-frequency 
limits  of 
\eqref{eq:main_svw} include the Hunter--Saxton 
equation with gradient noise:
\begin{align*}
0 &= \d q + \bk{\pd_x (u\, q)- \frac12 q^2 }\,\d t 
	+ \pd_x \bk{\sigma q}\circ \d W,  \quad 
q = \pd_x u,
\end{align*}
and, to second order, the Camassa--Holm equation 
with gradient noise:
\begin{align*}
0 &= \d u + \bk{u \pd_x u + \pd_x P}\,\d t 
	+ \sigma\,\pd_x u \circ \d W, \quad 
P = K*\bk{u^2 + \frac12 \bk{\pd_x u }^2}, 
\end{align*}
where $K$ is the Helmholtz kernel $\bk{1 - \pd_{xx}^2}^{-1}$ 
(see \cite[Section 2.3]{GHZ1997}, in the deterministic setting).
Dissipative $H^1$ weak solutions to these 
stochastic partial differential equations (SPDEs) were studied in
\cite{GHKP2022,Holden:2020aa, HKP2023} on $\T$ and 
with linear multiplicative noise on $\R$ in \cite{Chen:2021tr}.
The development of this paper follows 
in the vein of these preceding works. 
Related literature on PDE-strong solutions to 
stochastic Camassa--Holm 
equations is vast, and we mention local 
well-posedness results of \cite{Albeverio:2021uf} on $\R$ 
using Kato's operator theoretic methods  for gradient noise. 
There are corresponding well-posedness 
results for additive noise \cite{Chen:2012aa}, 
for multiplicative noise (see \cite{Tang:2018aa} and references there), 
gradient jump noise (see \cite{CL2022} and references there), 
and even pseudo-differential noises (see \cite{Tang2023} and references there).
A recent, first foray into stochastic variational wave equations is \cite{GV2023}, 
where the authors considered well-posedness 
on $\T$ with additive, cylindrical, It\^o noise. 
The additive nature of the noise there facilitates the 
transformation of the stochastic variational wave 
system in $(R,S)$-variables into a system 
of random PDEs.

Following \cite{GHKP2022,HKP2023}, 
in order to derive existence of dissipative martingale and 
PDE-weak solutions of \eqref{eq:main_svw} 
with general (low regularity) initial data in $L^2(\T)$, 
which we plan to carry out in subsequent work, 
a first step is to establish pathwise 
well-posedness to a viscous approximation. 
This is the goal of the present paper. 

The cross variation $\frac12 [\sigma \pd_x \bk{R + S}, W]$ 
is $\sigma \pd_x \bk{\sigma\,\pd_x \bk{R + S}}$ (sans extra 
$1/2$ factor). 
The formal conversion between Stratonovich 
and It\^o noise (which can be made rigorous by 
requiring that $t \mapsto 
	\int_{\T}\pd_x \varphi R\,\d x$ 
and $t \mapsto \int_{\T} 
	\pd_x \varphi S\,\d x $ 
be continuous, adapted semi-martingales for any $\varphi \in C^2(\T)$, 
see, e.g., \cite[page 1460]{AF2011}) 
then motivates us to study
\begin{equation}\label{eq:vvw}
\begin{aligned}
\d R 
 - \nu \pd_{xx}^2 R\,\d t
&=	 \pd_x\bk{c(u)\, R}\,\d t
	 -  \tilde{c}(u) \bk{R - S}^2\,\d t\\
&\quad\,\, +	\sigma\,  \pd_x \bk{R + S}\,\d {W} 
	+ \sigma \,\pd_x \bk{\sigma\, \pd_x\bk{R + S}}\,\d t, \\
\d S 
- \nu \pd_{xx}^2S\,\d t
	 &= 	- \pd_x \bk{c(u) S}\,\d t 
	 -  \tilde{c}(u) \bk{R- S}^2\,\d t\\
	&\quad\,\, +  \sigma\, \pd_x\bk{R + S}\,\d {W} 
	+  \sigma \,\pd_x \bk{\sigma\, \pd_x \bk{R + S}}\,\d t.
\end{aligned}
\end{equation}
Here we have used the  final ``constitutive" 
equation in \eqref{eq:main_svw}  relating 
$u$ and $(R, S)$ 
\begin{align}\label{eq:constitutive}
2c(u) \,\pd_x u = R - S
\end{align}
to write $c(u) \pd_x R + \tilde{c}(u) \bk{R^2 - S^2}$ 
in the conservative form as 
$ \pd_x \bk{c(u) R} - \tilde{c}(u) \bk{R - S}^2$ 
in the $R$-equation, and 
$- c(u) \pd_x S - \tilde{c}(u)\bk{R^2 - S^2}$ 
as 
$ -\pd_x \bk{c(u) S} - \tilde{c}(u) \bk{R - S}^2$ 
in the $S$-equation.

Our main theorem is:
\begin{thm}\label{thm:main}

Let $(R^0,S^0) \in\bk{L^2(\T)}^2$ have 
finite $2p_0 > 4$ moments, and be 
such that $\int_\T R^0 - S^0 \,\d x = 0$.
On Assumption \ref{sum:c_sigma}, the viscous 
variational wave equation with transport noise 
\eqref{eq:vvw} -- \eqref{eq:constitutive}, 
with initial condition $(R^0,S^0)$ 
has a unique pathwise solution $(R,S)$ in 
the sense of Definition \ref{def:str_sol}. 
Moreover, $R$ and $S$ have continuous 
paths in $L^2(\T)$, in fact $(R,S) \in 
	\bk{L^{\overline{p}}(\Omega;C([0,T];L^2(\T)))}^2$
for any $\overline{p} < 2 p_0$. 
\end{thm}

One important aspect of this paper is the 
way that \eqref{eq:constitutive} is solved on $\T$. 
Our key tool is the inverse $\pd_x^{-1}$ of the spatial derivative 
on zero average functions in $H^s_x$, $s \in \R$, on the periodic 
domain, defined in \cite[Equation 2.12]{Hol2010} 
(see \eqref{eq:dx_inverse} for details). 
Consider the anti-derivative of $c$:
\begin{align}\label{eq:F_defin}
F(u) := \int_0^u c(r)\,\d r.
\end{align}
Since $c$ is assumed to be uniformly 
positive, $F$ is strictly increasing and 
has an inverse, and $u$ can be readily 
recovered from $F(u)$. 

Let $q:= R - S$. By \eqref{eq:constitutive}, $q = 2\pd_x F(u)$, 
and hence has zero average over $\T$. 
We now use \eqref{eq:constitutive} 
and the inverse $\pd_x^{-1}$ to produce a 
candidate solution $u$ by writing: 
\begin{align}\label{eq:constitutive_solve}
u = F^{-1} (\frac12 \pd_x^{-1} q) 
= F^{-1}(\frac12\int_0^x q(t,y) \,\d y 
	- \frac{x}2\underbrace{\int_\T q(t,y)\,\d y}_{ = 0} 
	-  \tilde{h}(t)),
\end{align}
where we choose $h(t)$ so that $F(u(t,x))$ 
has zero spatial average thus:
$$
\tilde{h}(t) := \frac12 \int_\T \bigg[\int_0^y q(t,z)\,\d z 
		- y\underbrace{\int_\T q(t,z)\,\d z}_{= 0}
		\bigg]\,\d y.
$$
Since $\tilde{h}$ is not spatially dependent, \eqref{eq:constitutive} holds. 
Requiring $F(u)$ to have zero spatial average 
is a {choice}, akin to choosing the constant of 
integration. This choice is not entirely arbitrary; 
we elaborate on this later in Remark \ref{rem:u_construct_choice}. 
Were the problem posed on $\R$, 
we would have been able to integrate 
$q$ directly over $(-\infty, x]$ to obtain $F(u)$, 
by fixing $u(-\infty) = 0$ (as in \cite[Equation (1.8)]{HKR2009}).  

Having solved \eqref{eq:constitutive} thus, 
\eqref{eq:vvw} can be analysed using  
techniques inspired by \cite{ZZ2001a} and subsequent 
works by those authors in the context of variational wave equations, with further 
ingredients developed in the stochastic setting 
in \cite{GHKP2022, HKP2023} 
and in works cited there.  These methods are sketched out 
in finer detail in Section \ref{sec:strategy}.

\smallskip

\subsection{Viscosity}\label{sec:viscosity}

Weak solutions in the deterministic setting is 
usually studied via a limit of one-sided linear approximations 
to nonlinear terms in \eqref{eq:main_det2} by setting 
 \begin{align*}
\tilde{Q}_\ep(v) := \begin{cases}
\ep^{-1} \bk{v - \ep^{-1}} & v \ge \ep^{-1}\\
\frac12 v^2 & -\infty < v < \ep^{-1}
\end{cases},
\end{align*}
so that
$\tilde{Q}_\ep(R)$ replaces $R^2$ in the $R$-equation and 
$\tilde{Q}_\ep(S)$ replaces  $S^2$ in 
the $S$-equation in \eqref{eq:main_det2} 
\cite{ZZ2003,ZZ2005b}. We shall have need for 
similar linearisation in a two-level Galerkin 
approximation in Section \ref{sec:scheme}.

Viscous approximations, adding $\nu \pd_{xx}^2 R$ 
to the $R$-equation and $\nu\pd_{xx}^2 S$ to the 
$S$ equation was used to study local 
classical ($H^k_x$-) and global rarefactive 
solutions to \eqref{eq:main_det2} 
on the real line in \cite{ZZ2001a, ZZ2005b}. 
There, a Picard iteration 
was used to establish existence of solutions. 
In the stochastic setting, it can be easier, 
because of the stochastic integral, to consider 
Galerkin approximations, as was done for the 
viscous stochastic Camassa--Holm equation in \cite{HKP2023}.
Similar approximations are much more plentifully 
witnessed in the related stochastic Navier--Stokes 
and fluid equations literature. 
We explain the strategy of our well-posedness 
proof in the upcoming Section \ref{sec:strategy}. 

A direct substitution of the inviscid equation 
Riemann invariants \eqref{eq:riemann_inv} 
into the standard viscous approximation 
$$
\d v = c(u) \,\pd_x\bk{c(u)\pd_x u}\,\d t  
	+ \nu\pd_{xx}^2v \,\d t + \sigma \pd_x v \circ \d W
$$
will have given us viscous 
terms of the form $\nu\pd_{xx}^2 \bk{R + S}$ in 
both the $R$ and $S$ equations. 
The $L^2_tH^1_x$ 
inclusion for $\bk{R + S}$ arising from the (cross-)diffusion, 
but not for $R$ and $S$ separately, is insufficient 
for passing to the limit in Galerkin approximations 
of the nonlinear term $\tilde{c}(u)\bk{R - S}^2$ 
in \eqref{eq:vvw}. As we shall see in Remark 
\ref{rem:u_construct_choice}, a good choice 
for the viscosity is related to the construction 
\eqref{eq:constitutive_solve} of $u$. 

The need to extract limits in the nonlinear term 
to prove the existence of weak solutions with 
non-smooth initial data, i.e., data $R^0$ and $S^0$ 
which are $L^2(\T)$-valued, is a mathematical reason for 
studying the viscous equation at all.
Variational wave type equations characteristically 
possess {\em a priori} $L^{2 + \alpha}_{\omega, t, x}$ bounds 
for smooth solutions (and $c' > 0$).  Uniformly 
$L^{1 + \alpha/2}_{\omega, t, x}$-bounded approximations 
$\bk{R_N - S_N}^2$ of the nonlinear term can 
then be shown to converge 
weakly to a limit $\overline{\bk{R - S}^2}$, 
which by the uniform  bound is kept 
from becoming a measure in $(\omega, t, x)$.
The fact that this is not a measure 
is important for any subsequent 
renormalisation and propagation of 
compactness argument seeking to establish 
$\overline{\bk{R - S}^2} = \bk{R - S}^2$, $(\omega, t, x)$-a.e.
A similar strategy was pursued in the 
deterministic setting for the variational 
wave equation by e.g., \cite{ZZ2003,ZZ2005b}, 
and carried out in the much more delicate  
stochastic setting for the Camassa--Holm equation 
with gradient noise in \cite{GHKP2022,HKP2023}.

The existence of 
$L^{2 + \alpha}_{\omega, t ,x}$ bounds 
depends intimately on the structure of 
the nonlinearity in both equations of \eqref{eq:main_svw}.
These bounds are lost when projection operators 
in simple Galerkin approximations interfere 
with intricate algebraic manipulations leading 
to a uniform  estimate. 
The viscous approximation gives us 
$H^1_x (\hookrightarrow L^\infty_x)$  bounds 
by which Galerkin approximations can 
be shown to converge even in the nonlinearity, 
replacing the $L^{2 + \alpha}_{\omega, t, x}$ bounds 
at the viscous level. 
Viscous terms leave the structure of the nonlinearity 
intact so that uniform $L^{2 + \alpha}_{\omega, t, x}$ control 
can be exploited in a secondary limit 
(in this case, 
the vanishing viscosity limit to be dealt with in 
subsequent work). 
 Let us take the opportunity in mentioning 
subsequent work to point out that in the multidimensional 
setting, even the viscous, spherically symmetric problem 
corresponding to, e.g.~the deterministic equations studied 
in \cite{DHW2020} remains, as far as we know, open.

\subsection{Strategy and outline of paper}\label{sec:strategy}

In the remainder of this paper, 
we first give the precise definition of 
solutions in the following subsection. 
Afterwards, our general strategy 
adapts that of \cite{HKP2023} and 
other works referenced there.

In Section \ref{sec:scheme}, we construct a two-level 
Galerkin scheme with a cut-off function on the 
nonlinear term $\bk{R - S}^2$. This cut-off level defines 
a stopping time for which the cut-off free Galerkin 
approximations hold exactly. An  
additional limit needs to be taken to establish the 
well-posedness of the scheme where the cut-off is 
sent to $\infty$. As a part of this scheme, it is 
necessary to smooth out the nonlinear wave-speed 
$c$. In Sections \ref{sec:apriori_tightness} 
-- \ref{sec:u_N_tightness}, 
tightness of laws for the solutions of the 
Galerkin system is established. 
We prove energy estimates for the sequence 
of (cut-off free) Galerkin approximants 
$R_N$ and $S_N$ to $R$ and $S$, showing they have laws 
that are tight in the quasi-Polish spaces
$C([0,T];L^2(\T)-w)$ and $L^2([0,T]\times \T)$. 
After that, we show that approximations $u_N$ 
to $u$ (similarly constructed as 
\eqref{eq:constitutive_solve}) 
converge in law in $C([0,T] \times \T)$. 
As indicated, we elaborate on the construction 
\eqref{eq:constitutive_solve} in Section 
\ref{sec:galerkin_compactness} and in the concluding 
Remark \ref{rem:u_construct_choice} of this section.

In Section \ref{sec:SJthm}, 
the tightness results of Section 
\ref{sec:galerkin_compactness} are used via the 
Skorokhod--Jakubowski theorem to prove the 
existence of martingale solutions. Instead of the 
standard martingale identification argument 
of \cite{Brzezniak:2011aa} (and many others 
subsequently), we establish the approximating 
equations on a new probability space 
in Proposition \ref{thm:nth_eq_Nlimit} 
using measure-preserving maps introduced 
by Dudley \cite{Dud1985} via the version of the 
Skorokhod--Jakubowski theorem in 
\cite[Theorem A.1]{Punshon-Smith:2018aa}. 
We describe these maps 
in Theorem \ref{thm:skorokhod_N} and 
refer to them as {\em Dudley maps}.
The limit of these equations are then taken directly 
using the stochastic integral convergence 
lemma of Debussche--Glatt-Holtz--Temam 
\cite[Lemma 2.1]{Debussche:2011aa}.

After this, we prove pathwise 
uniqueness of strong solutions in Section 
\ref{sec:pathwise} using a stochastic 
Gronwall inequality at suitably chosen 
stopping times. Pathwise uniqueness 
is proven by comparing two {\em pairs} of equations. 
The result requires a delicate exploitation of the 
nonlinear structure in equations of 
\eqref{eq:vvw}, just as for deriving the energy inequality. 
Unlike the energy inequality, some of this 
structure is masked by the fact that 
there are two $R$-equations and two $S$-equations. 
New variational wave equation specific 
convergence estimates for the nonlinear 
composition $c(u)$ are needed and proven 
in both Sections \ref{sec:SJthm} and 
\ref{sec:pathwise}. We also require double 
commutator estimates beyond the standard 
DiPerna--Lions commutators rigorously to justify 
manipulations on the transport noise term. 

A quasi-Polish version of the 
Gy\"ongy--Krylov lemma \cite[Lemma 1.1]{GK1996}
will then imply that probabilistically 
strong solutions in fact exist. 

Finally, in Section \ref{sec:tempcont}, we 
show that the unique pathwise  
solutions $(R,S)$ take values not only in 
$\bk{L^\infty([0,T];L^2(\T))}^2$, but in fact a.s.~have  
continuous paths in $\bk{L^2(\T)}^2$, and lie in 
$\bk{L^{\overline{p}}(\Omega;C([0,T];L^2(\T)))}^2$ 
for any $\overline{p} < 2 p_0$. 
We show this improved inclusion by establishing 
that the mollified quantities 
$\{R *J_\delta\}_{\delta > 0}$ and 
$\{S*J_\delta\}_{\delta > 0}$, where 
$J_\delta$ is a spatial mollifier indexed 
by $\delta > 0$, are each Cauchy in 
$L^{\overline{p}}_\omega C_t L^2_x$. This proves the final claim 
of Theorem \ref{thm:main}.

\subsection{Definitions of solutions}
We use the following concepts of solutions 
to \eqref{eq:vvw} in subsequent sections. 

\begin{defin}[Martingale solutions] \label{def:mart_sol} 
Fix $p_0 > 2$. Let $ \Lambda$ be a probability 
measure with $2p_0$th moments on $\bk{L^2(\T)}^2$, 
supported on 
$\{(f,g) \in \bk{L^2(\T)}^2 : \int_\T f - g \,\d x = 0\}$, i.e., 
$\int_{\bk{L^2_x}^2} \norm{f}_{L^2}^{2p_0} + \norm{g}_{L^2}^{2p_0} \,\Lambda(\d f , \d g) < \infty$.
A quadruple $(R, S, E, W)$ 
is a weak martingale solution to \eqref{eq:vvw} 
with initial distribution $\Lambda$ if:
\begin{itemize}
\item[(i)] $E:=(\Omega, \mathcal{F},\{\mathcal{F}_t\}_{t \in [0,T]}, \mathbb{P})$ 
is a filtered probability space with a complete right-continuous filtration;

\item[(ii)] $W$ is a $\{\mathcal{F}_t\}_{t \in [0,T]}$-standard Brownian motion;

\item[(iii)] $(R, S)$ is adapted, and included in 
	the space 
	$\bk{L^{2p_0}(\Omega;L^\infty([0,T];L^2(\T)))}^2 \cap 
	 \bk{L^{2p_0}(\Omega;L^2([0,T];H^1(\T)))}^2$;

\item[(iv)] for any $\varphi \in C^1(\T)$, 
	the maps $t \mapsto \int_\T R \varphi \,\d x$ 
	and $t \mapsto \int_\T S \varphi \,\d x$ are 
	progressively measurable and $\mathbb{P}$-a.s.~continuous;

\item[(v)] the law of $(R^0, S^0) := (R(0), S(0))$ 
	on $\bk{L^2(\T)}^2$ is $\Lambda$;

\item[(vi)] For every $\varphi \in C^2(\T)$ 
	and every $t \in [0,T]$, 
	\eqref{eq:vvw} are a.s.~satisfied weakly, 
	with $u \in L^{2p_0}(C([0,T]\times \T))$ 
	given by \eqref{eq:constitutive_solve}.
	That is,
	\begin{align*}
	&\int \varphi \bk{R - R^0}\,\d x \\
	&= - \int_0^t \int_\T \pd_x \varphi\, c(u) R \,\d x\,\d t'
	- \int_0^t \int_{\T} \varphi\tilde{c}(u) \bk{R - S}^2\,\d x \,\d t'\\
	&\quad 	\,\,
	 -\nu \int_0^t \int_{\T} \pd_x\varphi  \,\pd_x R\,\d x \,\d t'
	 - \int_0^t \int_{\T} \pd_x\bk{\varphi \sigma }  \bk{R + S}\,\d x \,\d W\\
	&\quad  \,\,
	-  \int_0^t \int_{\T} \sigma \,
		\pd_x \bk{\varphi \sigma} \, \pd_x\bk{R + S}\,\d x \,\d t',
	\end{align*}
	and 
	\begin{align*}
	&\int \varphi \bk{S - S^0}\,\d x \\
	&= 
	\int_0^t \int_{\T} \pd_x \varphi\, c(u) S\,\d x\,\d t'
	 - \int_0^t \int_{\T} \varphi \tilde{c}(u) \bk{R - S}^2\,\d x \,\d t'\\
	&\quad  \,\,
		 -\nu\int_0^t \int_{\T} \pd_x \varphi \,\pd_x S\,\d x \,\d t'
	 - \int_0^t \int_{\T} \pd_x\bk{\varphi \sigma }  \bk{R + S}\,\d x \,\d W\\
	&\quad  \,\,
	-  \int_0^t \int_{\T} \sigma \,
		\pd_x \bk{\varphi \sigma} \, \pd_x\bk{R + S}\,\d x \,\d t'.
	\end{align*}

\end{itemize}
\end{defin}

\begin{defin}[Pathwise solutions]\label{def:str_sol}
Pathwise, or probabilistically strong, solutions $(R,S)$ 
to \eqref{eq:vvw} with initial conditions 
$(R^0, S^0) \in \bk{L^2(\T)}^2$ such that 
$\int_\T R^0 - S^0 \,\d x = 0$ a.s.~are martingale 
solutions with an initial distribution $\Lambda$ 
for which a fixed stochastic basis 
$((\Omega, \mathcal{F}, 
	\{\mathcal{F}_t\}_{t \in [0,T]}, \mathbb{P}), W)$ 
is given, the law of $(R^0, S^0)$ is $\Lambda$, 
and $(R,S,u)$ satisfy \eqref{eq:constitutive_solve} 
and the equations 
in (vi) of Definition \ref{def:mart_sol}.
\end{defin}

\section{Galerkin approximations}\label{sec:galerkin_compactness}

\subsection{The Galerkin scheme}\label{sec:scheme}
We now build our two-tiered Galerkin scheme. 

Let $\{e_1, e_2, \ldots\}$ be a complete 
orthonormal basis of $L^2(\T)$ contained 
in $H^3(\T)$. Let ${\bf P}_N$ be the $L^2(\T)$ 
orthogonal projection onto the subspace 
spanned by the first $2N - 1$ basis functions. 
By choosing $e_{2j} := \cos(2 \pi j x)$ and 
$e_{2j + 1} := \sin(2 \pi jx)$ to be the eigenfunctions 
of $\pd_x^2$ on the circle $\T$,  it can be checked
that ${\bf P}_N$ commutes with $\pd_x$ 
(e.g., \cite[Equation (4.2)]{HKP2023}). 
Let $R_N$ and $S_N$ be $\bk{2N - 1}$st 
order Galerkin approximants associated 
with $R$ and $S$, respectively. 

We seek to design a Galerkin scheme 
whose solution exists and is unique by 
general SDE well-posedness theorems 
(such as \cite[Theorem 2.9]{KS1998}, which we shall use). 
It is necessary then, to modify the 
equations \eqref{eq:vvw} not only by the 
projection ${\bf P}_N$, but also to iron out 
coefficients that fail to be Lipschitz in 
$R_N$ and $S_N$. There are two sources 
of non-Lipschitzness. The first is the 
coefficient $\tilde{c} = c/\bk{4c}$. Since $c$ is only 
once continuously differentiable by 
Assumption \ref{sum:c_sigma}, we cannot 
expect $\tilde{c}(u_N)$ to be Lipschitz in $(R_N, S_N)$ 
even if $u_N$ is so as a function of $(R_N,S_N)$.
A second reason that a direct projection of 
\eqref{eq:vvw} fails to be Lipschitz is the 
more obvious nonlinear factor $(R - S)^2$ there. 
We handle this by a truncation (indexed by $k$) 
which can be removed upon establishing 
{\em a priori} bounds (Lemma \ref{thm:galerkin_wp}). 

We first detail the way in which we handle 
the failure of $\tilde{c}$ to be Lipschitz.
Let $\{c_N\}$ be a sequence in $C^2(\R)$ such that 
\begin{align}\label{eq:c_N_approx}
c_N \to c \quad \text{in $C^1(\R)$}.
\end{align}
This implies that $c'_N/\bk{4c_N}=: \tilde{c}_N \to \tilde{c}$ 
in $C(\R)$.
We assume that $\kappa > 0$ in Assumption 
\ref{sum:c_sigma} has been chosen sufficiently 
large such that $c_N$ all take values in 
$[\kappa^{-1}, \kappa]$ and $\abs{c_N'}\le \kappa$.

Let $F$ be the anti-derivative of $c$ 
defined in \eqref{eq:F_defin}. We similarly  
define the anti-derivatives of $c_N$ to be 
\begin{align}\label{eq:F_N_defin}
F_N(u) := \int_0^u c_N(r)\,\d r.
\end{align}
Like $F$, $F_N$ is bi-Lipschitz (both $F_N$ 
and its inverse are Lipschitz), with Lipschitz 
constant for $F_N$ and $F_N^{-1}$ both 
bounded uniformly by $\kappa$.

In order to construct our approximations 
$u_N$, we use $F_N$ and the operator $\pd_x^{-1}$ 
alluded to in \eqref{eq:constitutive_solve} 
and described precisely below. 
For $f \in H^s(\T)$, we define 
the inverse operator $\pd_x^{-1}$ on $H^s(\T)$ 
following \cite[Equation (2.12)]{Hol2010}:
\begin{equation}\label{eq:dx_inverse}
\begin{aligned}
\bk{\pd_x^{-1} f}(x)
:= \int_0^x &\frac{f(y)}2 \,\d y 
	- x\int_\T \frac{f(y)}2 \,\d y\\
&	- \int_\T \bigg[\int_0^y \frac{f(z)}2 \,\d z
		- y \int_\T \frac{f(z)}2\,\d z\bigg]\,\d y.
\end{aligned}
\end{equation}
Let $H_0^s(\T)$ be the subspace 
of $H^s(\T)$ with zero spatial average. 
The operator $\pd_x^{-1}$ is both a 
left and a right inverse of $\pd_x$ on 
$H^s_0(\T)$, $s \in \R$, and is continuous  
$H^s_0(\T) \to H^{s + 1}_0(\T)$ \cite[Lemma 3]{Hol2010}.

We shall use the following convenient 
notation throughout the remainder of 
the paper. 
For $G,H \in C([0,T];L^2(\T)-w)$, set
\begin{equation}\label{eq:u_construct}
\begin{aligned}
\mathfrak{u}_N(G,H)(t,x) &:=  F_N^{-1} ( \pd_x^{-1} \frac{G - H}2),\\ 
\mathfrak{u}(G,H)(t,x) &:=  F^{-1} ( \pd_x^{-1} \frac{G - H}2).
\end{aligned}
\end{equation}
We discuss this construction in greater depth 
in Remark \ref{rem:u_construct_choice}.
Approximating $u$ by 
\begin{align}\label{eq:u_approx_defin}
u_N := \mathfrak{u}_N(R_N,S_N)
\end{align}
gives us 
\begin{align}\label{eq:RNSN_dFdx}
R_N - S_N = 2 \pd_x F_N(u_N) = 2c_N(u_N) \,\pd_x u_N,
\end{align}
which agrees with \eqref{eq:riemann_inv}.
As a consequence of the definition 
of $u_N$, we also have:
\begin{align}\label{eq:c_deriv}
\pd_x c_N(u_N) = c'_N(u_N) \pd_x u_N 
	= \frac{c'_N(u_N)}{2c_N(u_N)} \bk{R_N-S_N} 
	= 2\tilde{c}_N(u_N) \bk{R_N-S_N}.
\end{align}

Our Galerkin approximation is two-tiered 
in order to handle the non-Lipschitz nonlinearity. 
We first define a cut-off scheme where the $N$th 
order approximation has a cut-off indexed by $k$. 
For each fixed $N$, we shall first take $k \uparrow \infty$. 
The $N\uparrow \infty$ limit will be taken using 
Jakubowski's extension of Skorokhod's 
representation theorem in Section \ref{sec:SJthm}.
Therefore, for $k \in \N_{ \ge 1}$,  now 
let $Q_k: L^2(\T) \to L^2(\T)$ be the cut-off 
function (see, e.g., \cite[Equation (7)]{Flandoli:1995aa})
$$
Q_k(f) := \chi(\norm{f}_{L^2(\T)}) f^2, 
\quad \chi(r) := \begin{cases}
1 & \abs{r} \le k\\
0 & \abs{r} \ge k  + 1
\end{cases}, 
\quad \chi \in C^\infty(\R;[0,1]).
$$

Given $(R^0,S^0)$ with law $\Lambda$ on 
$\bk{L^2(\T)}^2$, we study the Galerkin system 
\begin{equation}\label{eq:galerkinQ}
\begin{aligned}
\d R_N &	 - \nu \pd_{xx}^2 R_N\,\d t\\
& =   {\bf P}_N\big[ c_N(u_N)  \pd_x R_N  \big]\,\d t
		 +   {\bf P}_N \big[\tilde{c}_N(u_N) 
			\bk{Q_k(R_N) - Q_k(S_N) }\big]\,\d t\\
&\quad\,\,+ {\bf P}_N \big[\sigma \pd_x \bk{R_N  + S_N}\big]\,\d W
		+ {\bf P}_N \big[\sigma \,
			\pd_x \bk{\sigma\,\pd_x \bk{R_N  + S_N}}\big] \,\d t,\\
\d S_N &	- \nu \pd_{xx}^2 S_N\,\d t\\
	&= -   {\bf P}_N\big[ c_N(u_N)  \pd_xS_N \big]\,\d t 
		 - {\bf P}_N \big[\tilde{c}_N(u_N) 
			\bk{Q_k(R_N) - Q_k(S_N)} \big]\,\d t\\
&\quad\,\,+ {\bf P}_N \big[\sigma \pd_x \bk{R_N  + S_N}\big]\,\d W
		+  {\bf P}_N \big[\sigma \,
			\pd_x \bk{\sigma\,\pd_x \bk{R_N  + S_N}}\big] \,\d t,
\end{aligned}
\end{equation}
appended with the intial conditions:
\begin{align}\label{eq:galerkin_init}
R_N(0) = R^0_{N}:= {\bf P}_N R^0,\qquad 
S_N(0) = S^0_{N} :=   {\bf P}_N S^0.
\end{align}
For any finite $k$, the system \eqref{eq:galerkinQ} 
is Lipschitz in $R_N$ and $S_N$, and by 
\cite[Theorem 2.9]{KS1998}, possesses
unique strong solutions $(R_{N,k}, S_{N,k})$. 
We seek to take the $k \uparrow \infty$ 
to establish well-posedness for the system 
\eqref{eq:galerkinQ} with $Q_k(v)$ 
replaced by $v^2$, we have the system 
(in equivalent divergence form using \eqref{eq:c_deriv}, since $N$ is finite):
\begin{equation}\label{eq:galerkin_lim}
\begin{aligned}
\d R_N &	 - \nu \pd_{xx}^2 R_N\,\d t\\
& =  \pd_x {\bf P}_N\big[ c_N(u_N)  R_N  \big]\,\d t
		 -   {\bf P}_N \big[\tilde{c}_N(u_N) 
			\bk{R_N - S_N}^2 \big]\,\d t\\
&\quad\,\,+ {\bf P}_N \big[\sigma \pd_x \bk{R_N  + S_N}\big]\,\d W
		+ {\bf P}_N \big[\sigma \,
			\pd_x \bk{\sigma\,\pd_x \bk{R_N  + S_N}}\big] \,\d t,\\
\d S_N &	- \nu \pd_{xx}^2 S_N\,\d t\\
	&= -  \pd_x {\bf P}_N\big[ c_N(u_N)  S_N \big]\,\d t 
		 - {\bf P}_N \big[\tilde{c}_N(u_N) 
			\bk{R_N - S_N}^2 \big]\,\d t\\
&\quad\,\,+ {\bf P}_N \big[\sigma \pd_x \bk{R_N  + S_N}\big]\,\d W
		+  {\bf P}_N \big[\sigma \,
			\pd_x \bk{\sigma\,\pd_x \bk{R_N  + S_N}}\big] \,\d t,
\end{aligned}
\end{equation}
with the same initial conditions 
\eqref{eq:galerkin_init}. Define the stopping time 
$$
\tau_{N,k} := \inf \{t > 0 :
	 \norm{R_{N,k}(t)}_{L^2(\T)}^2 + \norm{S_{N,k}(t)}_{L^2(\T)}^2 = k \}.
$$
For each fixed $N$, up to $\tau_{N,k}$,  
$(R_{N,k}, S_{N,k}, \tau_{N,k})$ is a 
local solution to \eqref{eq:galerkin_lim}. 
We now prove a uniform bound for 
local solutions: 
\begin{lem}\label{thm:galerkin_energy1}
Let $\tau$ be a stopping time and let 
$({R}_N,{S}_N)$ be a $\bk{H^3(\T)}^2$-valued 
continuous, adapted process that satisfies 
\eqref{eq:galerkin_lim} for every $t \le \tau$. Assume 
$\Ex \int_{\T} {R}_N^2(0) + {S}_N^2(0) \,\d x < \infty$. 
We have the uniform-in-$\tau$ bound:
\begin{align*}
\Ex \sup_{t \in [0,{T\wedge \tau}]}&
	 \int_{\T} {R}_N^2 + {S}_N^2 \,\d x \\
&+ \nu \Ex \int_0^{T \wedge \tau} \int_{\T} 
	\abs{ \pd_x {R}_N}^2 + \abs{\pd_x {S}_N}^2\,\d x \,\d t 
\lesssim_{{R}_N(0), {S}_N(0), \sigma, T} 1.
\end{align*}
\end{lem}
\begin{rem}\label{rem:sigma_reg}
It is possible to introduce an exponential factor 
$e^{CT/\nu}$ to the right hand side of the bound 
in the statement, in exchange for the reduced 
regularity requirement $\sigma \in W^{1,\infty}(\T)$ 
(cf. Assumption \ref{sum:c_sigma}). This is rougher 
than the standard assumption on $\sigma$. To do this, 
one simply keeps one derivative on $\bk{R + S}^2$ 
in \eqref{eq:st_energy_345}, and applies the Cauchy--Schwarz 
inequality to produce 
$\nu\norm{\pd_x \bk{R + S}}_{L^2_{t,x}}^2$, 
which can then be absorbed 
into the dissipation. We refrain from this here 
as we aim to study the inviscid limit in a forthcoming work.
\end{rem}

\begin{proof}

Since ${\bf P}_N R_N = R_N$, and the projection 
both is self-adjoint and commutes with the 
spatial derivative, we have upon integration-by-parts that:
$$
\nu \int_{\T} R_N  \,\pd_{xx}^2 R_N\,\d x
= -\nu \int_{\T} \abs{\pd_x R_N}^2\,\d x,
$$
and similarly for $S_N$ in place of $R_N$. 

We drop the subscripts on $u_N$, ${R}_N$, and ${S}_N$ 
in the remainder of the proof. 
Multiplying the first equation of \eqref{eq:galerkin_lim} by 
$R$ and the second equation by $S$, 
adding the two equations up and integrating 
in $x$, we have (by It\^o's formula)
\begin{equation}\label{eq:R_N2S_N2}
\begin{aligned}
\frac12\d \int_{\T} &\bk{R^2 + S^2} \,\d x \\
&+ \nu \int_{\T}{\abs{ \pd_x R}^2 + \abs{\pd_x S}^2}\,\d x \,\d t 
 =\sum_{j = 1}^5 I_i \,\d t  + I_6 \,\d W, 
\end{aligned}
\end{equation}
where 
\begin{align*}
I_1&:= \frac12 \int_\T c_N(u)\,\pd_x\bk{R^2 - S^2} \,\d x,\quad
I_2 := \int_{\T} \tilde{c}_N(u) \bk{R - S}\bk{R^2 - S^2}\,\d x, \\
I_3 &:=  \int_{\T} \abs{{\bf P}_N\big[\sigma\pd_x  \bk{R + S}\big]}^2\,\d x,\quad\,\,\,
I_4 := \int_{\T} R \sigma \,\pd_x\bk{\sigma \,\pd_x \bk{R + S}}\,\d x, \\
I_5 &:= \int_{\T} S \sigma \,\pd_x\bk{\sigma \,\pd_x  \bk{R + S}}\,\d x, \quad\,\,\,
I_6 :=  \frac12\int_\T \sigma \pd_x\bk{R + S}^2\,\d x,
\end{align*}
and $u$ is defined from $R$ and $S$ as 
in \eqref{eq:u_approx_defin}. 

Integrating by parts and using 
\eqref{eq:c_deriv}, $I_1 + I_2 = 0$.
Integrating by parts in $I_4$ and in $I_5$, 
and appealing to Bessel's inequality,
\begin{equation}\label{eq:st_energy_345}
\begin{aligned}
&I_3 + I_4 + I_5\\
& =  I_3 -  \int_{\T}\sigma^2
	 \abs{\pd_x R + \pd_x S}^2\,\d x 
 	+ \frac12\int_{\T}\pd_x 
	\bk{ \sigma \,\pd_x \sigma} \bk{R + S}^2\,\d x \\
&\le \frac12\norm{\sigma^2}_{W^{2,\infty}_x} 
	\bk{\norm{R}_{L^2_x}^2 + \norm{S}_{L^2_x}^2}.
\end{aligned}
\end{equation}

By the Burkholder--Davis--Gundy (BDG) inequality and Young's inequality,
\begin{align*}
\Ex\abs{ \int_0^{T\wedge \tau} I_6 \,\d W} 
&\le 2\norm{\pd_x \sigma}_{L^\infty_x} 
	\Ex \bk{\int_0^T \bk{\norm{R}_{L^2_x}^2 	
		+ \norm{S}_{L^2_x}^2}^2\,\d s }^{1/2}\\
&\le C_\sigma 
	\Ex \bk{\int_0^{T\wedge \tau} \bk{\norm{R}_{L^2_x}^2 	
		+ \norm{S}_{L^2_x}^2}\,\d s } \\
&\qquad + \frac14 \Ex \sup_{t \le {T\wedge \tau}}  \bk{\norm{R(t)}_{L^2_x}^2 	
		+ \norm{S(t)}_{L^2_x}^2}.
\end{align*}
The final term is absorbed into the left side of the inequality. 
Therefore,
\begin{align*}
\frac14\Ex \sup_{t \le {T\wedge \tau}} & \bk{\norm{R(t)}_{L^2_x}^2 	
		+ \norm{S(t)}_{L^2_x}^2}\\
&	\le  \Ex \bk{\norm{R(0)}_{L^2_x}^2 	
		+ \norm{S(0)}_{L^2_x}^2} \\
&\qquad		+ C_\sigma \Ex  \int_0^T \one{t \le {T\wedge \tau}}
	\bk{\norm{R}_{L^2_x}^2 + \norm{S}_{L^2_x}^2}\,\d t.
\end{align*}
Gronwall's inequality establishes the lemma.

\end{proof}

Since $({R}_N(0), {S}_N(0))$ have 
$2p_0$th moments in $L^2(\T)$, by It\^o's 
formula, we can derive from \eqref{eq:R_N2S_N2}
the following higher moment bound 
(such as \cite[Lemma 4.2]{HKP2023}):
\begin{align*}
\Ex \sup_{t \in [0,{T\wedge \tau}]}&
	\bk{ \int_{\T} {R}_N^2 + {S}_N^2 \,\d x}^{p_0} \\
&+ \nu^{p_0} \Ex \bk{ \int_0^{T \wedge \tau} \int_{\T} 
	\abs{ \pd_x {R}_N }^2+ \abs{\pd_x {S}_N}^2\,\d x \,\d t}^{p_0} 
\lesssim_{{R}_N(0), {S}_N(0), \sigma, T} 1.
\end{align*}

We now show that for each $N$,  
$\tau_{N,k}$ tends a.s.~to $T$ as $k \uparrow \infty$. 
A similar argument can be found in \cite[page 74]{Fla2008}.
\begin{prop}[Well-posedness of the 
	$k = \infty$ Galerkin scheme]\label{thm:galerkin_wp}
For each fixed $N$, let $(R_N, S_N)$ 
be the uniquely defined process which is equal 
to $(R_{N,k}, S_{N,k})$ on $[0, \tau_{N,k})$ 
for every $k$. Then $(R_N, S_N)$  is the unique 
strong solution to \eqref{eq:galerkin_lim} on $[0,T]$.

Moreover, 
\begin{equation}\label{eq:galerkin_L2}
\begin{aligned}
\Ex \sup_{t \in [0,T]}&\bk{
	 \int_{\T} R_N^2 + S_N^2 \,\d x}^{p_0}\\
&+ \nu^{p_0} \Ex \bk{\int_0^T \int_{\T}\abs{ \pd_x R_N }^2
			+ \abs{\pd_x S_N}^2\,\d x \,\d t}^{p_0} 
\lesssim_{R^0_N, S^0_N, \sigma, T} 1.
\end{aligned}
\end{equation}
\end{prop}

\begin{proof}
From Lemma \ref{thm:galerkin_energy1}, 
$$
\Ex \sup_{t \in [0,T]} \bk{
	\norm{R_{N,k}(t \wedge \tau_{N,k})}_{L^2_x}^2
	+ \norm{S_{N,k}(t \wedge \tau_{N,k})}_{L^2_x}^2}
\lesssim_{N,\sigma, T} 1. 
$$
Specifically, 
\begin{align*}
&k \mathbb{P}(\{\tau_{N,k} \le T\}) \\
&	= \Ex \bk{\one{\tau_{N,k} \le T}\bk{
	\norm{R_{N,k}(t \wedge \tau_{N,k})}_{L^2_x}^2
	+ \norm{S_{N,k}(t \wedge \tau_{N,k})}_{L^2_x}^2}}
	\lesssim_{N,\sigma, T} 1.
\end{align*}

Let $\tau_{N, \infty}$ be the time of 
existence of $(R_N, S_N)$. By definition 
of $R_N, S_N$, $\tau_{N,\infty} \ge \tau_{N,k}$ 
for any $k \in \mathbb{N}$. Therefore
$$
\mathbb{P}(\{\tau_{N,\infty} < T\}) 
	\le \mathbb{P}(\{\tau_{N,k} < T\}) 	
	\lesssim k^{-1} \to 0.
$$
The $L^{2p_0}(\Omega;L^\infty([0,T];L^2(\T))) 	
\cap L^{2p_0}(\Omega;L^2( [0,T];H^1(\T)))$ bound 
is inherited directly from Lemma 
\ref{thm:galerkin_energy1}.
\end{proof}

\subsection{{\em A priori} estimates and tightness}\label{sec:apriori_tightness}
Our present concern is the limit $N\uparrow \infty$ 
in $(R_N,S_N)$ to $(R, S)$. 
We first prove that the laws of 
$R_N$ and $S_N$ are tight in $L^2([0,T] \times \T)$. 
We prove further 
estimates that allow us to enforce 
the convergence in law in $C([0,T];L^2(\T)-w)$. 
Tightness of laws for 
$u_N$ in $C([0,T] \times \T)$ is proven in Section \ref{sec:u_N_tightness}.
%
We shall use the Aubin--Lions lemma repeatedly:
\begin{lem}[Aubin--Lions lemma 
	{\cite[Theorem 5, Corollary 4]{Simon:1987vn}}]\label{thm:AL}
Let $B_0 \doublehookrightarrow B 
	\hookrightarrow B_1$ be a sequence of 
Banach spaces for which $B_0$ and $B_1$ are 
reflexive. For $p \in (1,\infty)$ and $\alpha \in (0,1]$, 
\begin{align*}
L^p([0,T];B_0) \cap W^{\alpha,p}([0,T];B_1) 
	\doublehookrightarrow L^p([0,T];B).
\end{align*}
For $p = \infty$, $\alpha = 1$, and $r > 1$, 
\begin{align*}
L^\infty([0,T];B_0) \cap \dot{W}^{1,r}([0,T];B_1) 
	\doublehookrightarrow C([0,T];B).
\end{align*}
\end{lem}

Recall that tightness of laws for a sequence 
of variables  such as $\{R_N\}$ in a 
topological space $\mathcal{X}_1$ 
means that for every $\ep > 0$, there exists 
a compact $K_\ep \subset \mathcal{X}_1$ 
such that $\mathbb{P}(\{R_N \in K_\ep^c\}) \le \ep$, 
uniformly in $N$.

The first assertion of 
Lemma \ref{thm:AL} identifies 
compact sets that give us tightness of 
laws for $\{R_N\}_{N\ge 1}$, $\{S_N\}_{N\ge 1}$ 
in $L^p([0,T];B)$.
The compact sets are balls of finite radii 
in each of the spaces whose intersection 
embed compactly into $L^p([0,T];B)$.
Applying Lemma \ref{thm:AL} then requires 
two uniform bounds in $L^p([0,T];B_0)$ 
and in $W^{\alpha,p}([0,T];B_1)$ for suitably 
chosen indices and space $B_0$, $B_1$, 
and $B$, up to 
some moment in the probability variable.

Lemma \ref{thm:galerkin_energy1} provided 
one uniform bound in $L^p([0,T];B_0)$ with 
$p = 2$, $B_0 = H^1(\T) 
	\doublehookrightarrow B = L^2(\T)$. 
The second required uniform bound in 
$H^\alpha([0,T];B_1 = H^{-3}(\T))$ is available 
via a lemma we establish presently.
\begin{lem}\label{thm:galerkinCH-3}
Let $(R_N,S_N)$ be the unique strong 
solution to \eqref{eq:galerkin_lim}  with initial 
condition $(R^0_N, S^0_N)$. 
For any $\gamma < \bk{p_0/2 - 1}/p_0$, 
we have the uniform-in-$N$ bound:
\begin{align*}
\Ex \norm{R_N}_{C^\gamma([0,T];H^{-3}(\T))}^{p_0} 
+ \Ex \norm{S_N}_{C^\gamma([0,T];H^{-3}(\T))}^{p_0}
\lesssim 1
\end{align*}
This bound is also uniform in $\nu$. 
\end{lem}

\begin{proof}
The $H^{-3}(\T)$ norm is defined by
$$
\norm{f}_{H^{-3}(\T)} 
	:= \sup_{\norm{\varphi}_{H^3(\T)}\le 1}
		 \int \varphi f \,\d x .
$$

Consider first the equation for $R_N$. 
We integrate \eqref{eq:galerkin_lim} from $s$ to $t$, 
$s \le t$.
\begin{align*}
R_N(t) - R_N(s) = \sum_{j = 1}^5 I_j, 
\end{align*}
where
\begin{align*}
I_1 &:=  \int_s^t  \pd_x {\bf P}_N\big[ c_N(u_N) R_N  \big]\,\d r, \qquad \qquad\qquad
I_2 :=  \nu \int_s^t   \pd_{xx}^2 R_N\,\d r,\\
I_3 &:=  - \int_s^t  {\bf P}_N \big[\tilde{c}_N(u_N) 
			\bk{R_N - S_N}^2 \big]\,\d r,\\
I_4 &:=  \int_s^t {\bf P}_N \big[\sigma \,
			\pd_x \bk{\sigma\,\pd_x  \bk{R_N + S_N}}\big] \,\d r,\qquad
I_5 :=  \int_s^t  {\bf P}_N \big[\sigma \pd_x  \bk{R_N + S_N}\big]\,\d W.
\end{align*}

For any $\varphi \in H^3(\T)$ with 
$\norm{\varphi}_{H^3(\T)}\le 1$, Bessel's 
inequality implies $\norm{{\bf P}_N\varphi}_{H^3(\T)} \le 1$ also. 
By Sobolev embedding the first and 
second derivatives of ${\bf P}_N \varphi$ 
are bounded by $1$ in $L^\infty(\T)$. 

By the smoothness and uniform positivity of $c$, 
$\tilde{c} = c'/4c$ is a bounded function. Using 
\eqref{eq:c_deriv}, we can then estimate:
\begin{align*}
\int_{\T} \varphi I_1 \,\d x
	&= \int_s^t \int_{\T } \pd_x {\bf P}_N\varphi\,  c_N(u_N)  R_N\,\d x \,\d r\\
	&= \abs{t - s}\norm{c_N(u_N) \pd_x {\bf P}_N \varphi}_{L^\infty_{t,x}} 
	\norm{R_N}_{L^\infty_t L^1_x}.
\end{align*}
We now take the supremum over $\varphi$ 
in the unit ball of $H^3(\T)$, integrating 
in the probability variable and 
using the $L^2$ bound \eqref{eq:galerkin_L2}. 
With $p \le p_0$, we then have:
\begin{align}\label{eq:Calpha_interim1}
\Ex \norm{I_1}_{H^{-3}(\T)}^p \lesssim \abs{t - s}^p, 
\quad \text{uniformly in $N$}.
\end{align}

For $I_3$, we have
\begin{align*}
\int_\T \varphi I_3\,\d x 
&= -\int_s^t \int {\bf P}_N \varphi \tilde{c}_N(u) \bk{R_N - S_N}^2\,\d x\,\d r\\
& \le \kappa \abs{t - s} 
	\norm{{\bf P}_N \varphi\,\tilde{c}_N(u)}_{L^\infty_{t,x}} 
	\norm{R_N - S_N}_{L^\infty_tL^2_x}.
\end{align*}
And hence,
\begin{align}\label{eq:Calpha_interim3}
\Ex \norm{I_3}_{H^{-3}(\T)}^p\lesssim \abs{t - s}^{p}.
\end{align}

The integrals $I_2$ and $I_4$ are similar here, 
and we just present the calculations for $I_4$: 
\begin{align*}
\int_\T \varphi I_4 \,\d x 
&= - \int_s^t  \int_\T \sigma \pd_x \bk{\sigma {\bf P}_N \varphi} 
	\,\pd_x \bk{R_N + S_N}\,\d x \,\d r\\
& \le \abs{t - s}^{1/2}
	 \norm{\sigma \pd_x \bk{\sigma{\bf P}_N \varphi }}_{L^\infty_{t,x}} 	
	\norm{\pd_x\bk{R_N + S_N}}_{L^2_{t,x}}.
\end{align*}
The bounds for  $I_2$ 
can be handled similarly by transferring 
up to two spatial derivatives onto $\varphi$, 
so that uniformly in $N$ (and in $\nu$), 
\begin{align}\label{eq:Calpha_interim2}
 \Ex \norm{I_2}_{H^{-3}(\T)}^p
+ \Ex \norm{I_4}_{H^{-3}(\T)}^p\lesssim \abs{t - s}^{p/2}.
\end{align}
Using a Hilbert space-BDG inequality (see, e.g., \cite{MR2016}), 
\begin{align*}
&\Ex  \norm{ \int_s^t {\bf P}_N \big[ 
	\sigma \,\pd_x R_N\big] \,\d W}_{H^{-3}(\T)}^p\\
&\lesssim \Ex \bk{ \int_s^t \bk{\sup_{\norm{\varphi}_{H^3} \le 1}
	\int_{\T} \varphi {\bf P}_N 
	\big[ \sigma \,\pd_x R_N\big] \,\d x }^2\,\d r}^{p/2}\\
& \le \sup_{\norm{\varphi}_{H^3(\T)}\le 1} 
	\norm{\pd_x \bk{\sigma {\bf P}_N \varphi}}_{L^\infty_x}
	\Ex \norm{R_N}_{L^\infty_tL^1_x}^p \abs{t - s}^{p/2}
\lesssim \abs{t - s}^{p/2}.
\end{align*}
Here we used $\pd_x \sigma \in L^\infty(\T)$. 
This bound is uniform in $N$. Together with 
\eqref{eq:Calpha_interim1} -- \eqref{eq:Calpha_interim2}, 
\begin{align*}
\Ex \norm{R_N(t) - R_N(s)}_{H^{-3}(\T)}^p \lesssim \abs{t - s}^{p/2}.
\end{align*}
The same inequality holds with $R_N$ replaced by $S_N$. 

With $p$ upper bounded by $p_0$, 
the lemma holds by Kolmogorov's continuity criterion.
\end{proof}

By Lemma \ref{thm:AL}, the set 
$$
K_M := \{ w \in L^2([0,T] \times \T) : 
	\norm{w}_{H^\gamma([0,T];H^{-3}(\T))} 
	+ \norm{w}_{L^2([0,T];H^1(\T))} \le M\}
$$
is compact in $L^2([0,T]\times \T)$. 
Since $C^\gamma([0,T];H^{-3}(\T)) 
	\hookrightarrow H^\gamma([0,T];H^{-3}(\T))$, 
by Markov's inequality, 
\begin{equation}\label{eq:tightness_CalphaH-3}
\begin{aligned}
\mathbb{P}(R_N \in K_M^c) \le M^{-p_0}
	\Big( \Ex &\norm{R_N}_{C^\gamma([0,T];H^{-3}(\T))}^{p_0} \\
	&+ \Ex \norm{R_N}_{L^2([0,T];H^1(\T))}^{p_0}\Big) \lesssim M^{-p_0}.
\end{aligned}
\end{equation}
The bound is uniform in $N$. 
This is tightness of laws in  $L^2([0,T]\times \T)$ 
for $\{R_N\}$. Tightness of laws for $S_N$ 
in $L^2([0,T]\times \T)$ can be derived in the same way.

We are also keen to establish the tightness of laws 
for $R_N$ and $S_N$ in the time-continuous 
space $C([0,T];L^2(\T)-w)$ with values weakly 
in $L^2(\T)$ (see \cite[Appendix C]{Lions:NSI}). In 
the $N\uparrow \infty$ limit, the equations \eqref{eq:galerkin_lim} 
will need to be interpreted weakly. The inclusion of 
the limits in $C([0,T];L^2(\T)-w)$ will help us interpret 
the ``time derivative" terms in the sense of It\^o. 
We already have the means to show tightness in 
this space via the following compactness criterion:
\begin{lem}[{\cite[Corollary B.2]{Ondrejat:2010aa}}]
\label{thm:C_tH-w_embed}
Let $\alpha > 0$, $p,q \in (1,\infty)$ and 
$s \le k$ satisfy $p^{-1} - q^{-1} \le d^{-1}\bk{k - s}$. 
We have the compact embedding:
\begin{align*}
L^\infty([0,T];W^{k,p}(\T)) 
	\cap C^{\alpha}([0,T];W^{s,q}(\T)) 
\doublehookrightarrow C([0,T];W^{k,p}(\T)-w).
\end{align*}
\end{lem}

Applying Lemma \ref{thm:C_tH-w_embed} with 
$p = q = 2, k = 0, s = -3, \alpha = \gamma$, 
we find from Proposition \ref{thm:galerkin_wp} 
and Lemma \ref{thm:galerkinCH-3} the desired 
bounded inclusions implying tightness of laws 
for $\{R_N\}$ and $\{S_N\}$ in $C([0,T];L^2(\T)-w)$ 
via a similar calculation to \eqref{eq:tightness_CalphaH-3}.

Summarising the foregoing tightness statements, we have 
\begin{prop}[Tightness of laws on intersection space]\label{thm:tightness_N}
The laws of $\{R_N\}$ and the laws of $\{S_N\}$ are tight 
in $L^2([0,T]\times \T)$ and in $C([0,T];L^2(\T)-w)$. 
\end{prop}

\subsection{Tightness of laws for $\{u_N\}$}\label{sec:u_N_tightness}

We now establish tightness for the laws of $u_N$. 
Recall the construction 
$\mathfrak{u}(R_N,S_N)$ of $u_N$ 
in \eqref{eq:u_construct} and 
the definition \eqref{eq:F_defin} of 
the auxiliary function $F$ used in it.
\begin{lem}\label{thm:tilde_u}
The laws of $\{u_N\}$ constructed in 
\eqref{eq:u_approx_defin} are tight 
on $C([0,T]\times \T)$.
\end{lem}
\begin{proof}
We shall show that $\{u_N\}$ are bounded in 
$L^{2}(\Omega;L^\infty([0,T];H^1(\T)))$ 
and that $\{\pd_t u_N\}$ are 
bounded in $L^{2}(\Omega \times [0,T]\times \T)$. 
By the second statement of Lemma \ref{thm:AL}, 
we can conclude that the laws of $u_N$ are tight 
in $C([0,T];H^\beta(\T))$ for any $\beta < 1$.
In particular, tightness of laws hold for 
some $\beta > \frac12$, which facilitates the 
continuous embedding 
$H^\beta(\T) \hookrightarrow C(\T)$. 

By Assumption \ref{sum:c_sigma}, 
$c$ is lower bounded by $\kappa^{-1} > 0$. 
From \eqref{eq:RNSN_dFdx}, 
and the energy bound in Lemma \ref{thm:galerkin_wp},
\begin{align}\label{eq:RNSN_dFdx_unifbd}
\pd_x u_N
	= \frac1{c_N({u}_N)}\pd_x F({u}_N) 
	= \frac{{R}_N - {S}_N}{2c_N({u}_N)} 
	\in_b L^{2p_0}(\Omega;L^\infty([0,T];L^2(\T))),
\end{align}
where ``$\in_b$" denotes bounded inclusion over $N$. 

Next we treat the temporal derivatives $\pd_t u_N$. 
In the temporal direction, using the definition 
\eqref{eq:F_defin} of the auxiliary function $F$, we have
\begin{align}\label{eq:dF(u)1}
c_N({u}_N) \,\pd_t {u}_N = \pd_t F_N({u}_N)  = \pd_t \pd_x^{-1} \frac{R_N - S_N}2.
\end{align}

From \eqref{eq:galerkin_lim}, 
\begin{align*}
\pd_t \bk{R_N - S_N} 
	= \pd_x {\bf P}_N \big[ c_N(u_N) \bk{R_N + S_N} \big] 
	 +\nu \pd_{xx}^2 \bk{R_N + S_N}.
\end{align*}

Therefore, from the construction
\eqref{eq:u_construct} and the definition 
of the operator \eqref{eq:dx_inverse}, 
\begin{align*}
&\pd_t \pd_x^{-1} \bk{R_N - S_N} \\
& =  \int_0^x \pd_t\bk{R_N(t,y) - S_N(t,y)}\,\d y 
	- y \int_\T \pd_t \bk{R_N(t,z) - S_N(t,z)}\,\d z\\
&\qquad  - \int_\T \bigg[\int_0^y \pd_t \bk{R_N(t,z) - S_N(t,z)}\,\d z 
	-  y\int_\T  \pd_t \bk{R_N(t,z) - S_N(t,z)}\,\d z\bigg]\,\d y\\
& = \underbrace{{\bf P}_N \big[ c_N(u_N) \bk{R_N + S_N} \big]}_{=: I_1} 
	 +\underbrace{\nu \,\pd_x \bk{R_N + S_N}}_{=: I_2}
	 - \underbrace{\int_\T  {\bf P}_N 
	 	\big[ c_N(u_N) \bk{R_N + S_N} \big] \,\d y}_{=: I_3}.
\end{align*}

Putting this back into \eqref{eq:dF(u)1}, 
we have 
\begin{align*}
\pd_t u_N = \frac1{2c_N(u_N)} \bk{I_1 + I_2 - I_3}.
\end{align*}
We now establish the uniform bound 
$\{\pd_t u_N\} \subset_b 
	L^{2p_0}(\Omega;L^2([0,T] \times\T))$. 
First, by the lower bound on $c$ and 
Bessel's inequality,
\begin{align*}
\norm{ \frac1{c_N(u_N)} I_1}_{L^2_{t,x}} 
&\le \norm{\frac1{c_N(u_N)}}_{L^\infty_{t,x}} 
	\norm{{\bf P}_N \bk{R_N + S_N}}_{L^2_{t,x}}\\
&\le \kappa 	\norm{R_N + S_N}_{L^2_{t,x}} 
\overset{\eqref{eq:galerkin_L2}}{\in_b} L^{2p_0}(\Omega).
\end{align*}
Similarly,
\begin{align*}
\norm{ \frac1{c_N(u_N)} I_2}_{L^2_{t,x}}  
\lesssim_\kappa \norm{\pd_x \bk{R_N + S_N}}_{L^2_{t,x}} 
\overset{\eqref{eq:galerkin_L2}}{\in_b} L^{2p_0}(\Omega).
\end{align*}
Finally, 
\begin{align*}
\norm{ \frac1{c_N(u_N)} I_3}_{L^2_{t,x}} 
\lesssim_\kappa \abs{I_3} 
\lesssim_\kappa \norm{R_N + S_N}_{L^2_{t,x}}
\overset{\eqref{eq:galerkin_L2}}{\in_b} L^{2p_0}(\Omega).
\end{align*}

Therefore we find
\begin{align*}
\Ex \norm{\pd_t u_N}_{L^2_{t,x}}^{2p_0} 
&	\lesssim 
	\Ex\norm{{R}_N + {S}_N}_{L^2_t H^1_x}^{2p_0}
	\overset{\eqref{eq:galerkin_L2}}{\lesssim}_\nu 1.
\end{align*}

Along with \eqref{eq:RNSN_dFdx_unifbd}, we 
can invoke the second statement of 
Lemma \ref{thm:AL} with $r = 2$, 
$B_0 = H^1(\T)$, $B_1 = L^2(\T)$, and $B = H^{\beta}(\T)$ 
to conclude. We remark that we also have
\begin{equation}\label{eq:u_highermome}
\begin{aligned}
&\Ex \norm{u_N}_{C([0,T];H^\beta(\T))}^{2p_0} \\
&\qquad \lesssim \Ex\norm{u_N}_{L^\infty([0,T];H^1(\T))}^{2p_0} 
	+ \Ex \norm{u_N}_{\dot{H}^1([0,T];L^2(\T))}^{2p_0} \lesssim_{\nu,T} 1.
\end{aligned}
\end{equation}

\end{proof}

\begin{rem}[The operator $\pd_x^{-1}$ 
and the construction $\mathfrak{u}(R,S)$]
\label{rem:u_construct_choice}
We close this section with a remark on the 
construction \eqref{eq:u_construct} and its role 
in our compactness argument.

As mentioned preceding Theorem \ref{thm:main}, 
the construction \eqref{eq:u_construct} involved 
a choice of constant of integration. We explain 
this via a heuristic calculation. By integrating 
\eqref{eq:constitutive} directly, for $x,y \in \T$ we have:
\begin{align}\label{eq:F(u)_diffx}
F(u(t,x)) = F(u(t,y)) + \int_y^x \frac{R(t,z) - S(t,z)}{2}\,\d z.
\end{align}
Following \cite{GV2023}, 
we can write $u(t,y)$ 
as $u(t,y) = u^0(y) + \int_0^t \pd_t u(s,y)\,\d s$, 
where $u^0$ is an initial datum. 
By differentiating $\bk{R - S}$ in \eqref{eq:F(u)_diffx} 
in $t$, we can use \eqref{eq:vvw} to get:
\begin{align*}
0 = \bigg[2c(u(t,z)) \,\pd_t u(t,z)
	- c(u(t,z)) \bk{R + S}(t,z)
	- \nu \pd_x \bk{R - S} (t,z)\bigg]\bigg|_{z = y}^x.
\end{align*}
One then deduces that the expression 
in the bracket foregoing is independent 
of the spatial variable, and for any $z \in \T$,
\begin{align}\label{eq:constant_of_integration}
h(t) =2 c(u(t,z)) \,\pd_t u(t,z) 
	- c(u(t,z)) \bk{R + S}(t,z) 
	-\nu \pd_z \bk{R - S} (t,z).
\end{align}
for some process $h$. In the inviscid ($\nu = 0$), 
additive noise case, \cite{GV2023} 
used the choice $h(t) = 0$ in their well-posedness arguments. 

Suppose we replaced the viscosity $\nu\pd_{xx}^2 R$ 
in the $R$-equation by $\nu \pd_x \bk{c(u) \pd_x R}$
and $\nu\pd_{xx}^2 S$  by $\nu \pd_x \bk{c(u) \pd_x S}$ in the $S$-equation. 
Physically, these viscous terms model greater dissipation 
at higher wave-speeds. The lower boundedness 
of $c$ allows us to derive the same $L^2_tH^1_x$ 
bounds on $R$ and $S$. Where $h(t) = 0$, these viscous terms  
allow us to divide \eqref{eq:constant_of_integration} 
through by $c(u(t,z))$, to get an expression 
for $\pd_t u$ independent on $u$ itself, which 
can be inserted back into \eqref{eq:F(u)_diffx}, 
giving us (cf.  \cite[Equation (2.7)]{GV2023})
\begin{align*}
u(t,x) = F^{-1}(F(u^0(y) 
	+ \int_0^t \frac{R(s,y) + S(s,y)}2 &+ \nu \frac{\pd_x\bk{R - S}(s,y)}2 \,\d s )\\
	&+ \int_y^x \frac{R(t,z) - S(t,z)}2\,\d z). 
\end{align*}
This solves \eqref{eq:constitutive}. 
Galerkin approximations to this expression do 
not converge readily, however, 
as they involve the pointwise evaluation 
of $\frac\nu2\pd_x \bk{R - S}$ at the (arbitrary) 
spatial point $y$, whilst the derivative $\pd_x \bk{R - S}$, 
is only bounded a.s.~in $L^2_{t,x}$ in the limit 
(to be constructed via the Skorokhod theorem).

An alternative would be to choose 
viscosities such as $\nu \pd_{xx}^2 \bk{R + S}$ 
to be the same in both the $R$- and $S$-equations.
Then the viscous term is absent in $\pd_t \bk{R - S}$. 
But cross-diffusion gives us little control 
for passing to the limit approximations 
of nonlinear terms such as $\tilde{c}(u)\bk{R - S}^2$ 
in \eqref{eq:vvw}, defeating the purpose 
of studying the viscous approximation altogether. 

Our choice for $u$ 
is also a choice for a certain $h$ in 
\eqref{eq:constant_of_integration}. It  
reflects the fact that we have  
inverted the derivative on the periodic 
domain $\T$ for zero spatial average 
functions in a natural way using 
$\pd_x^{-1}$ defined in \eqref{eq:dx_inverse}. 
With $u = \mathfrak{u}(R,S)$ and viscous 
terms as in \eqref{eq:vvw}, differentiating 
$F(u)$ in $t$, we find
\begin{align*}
c(u) \pd_t u &= \frac12\Big[c(u(t,x)) \bk{R(t,x) +  S(t,x)} 
	+ \nu \pd_x \bk{R - S}(t,x)\Big]\\
&\quad\,\,- \underbrace{\frac12 \int_\T 
	\Big[c(u(t,y)) \bk{R(t,y) +  S(t,y)}\Big] \,\d y}_{ = \frac12 h(t)}.
\end{align*}
With this choice, there is no spatial pointwise 
evaluation in $\mathfrak{u}(R,S)$, and consequently 
passing to the 
limit for its approximants is more straightforward 
(see Lemma \ref{thm:c_convergences_N}). This choice is also 
less sensitive to the exact form of the viscosity.

\end{rem}

\section{The Skorokhod argument}\label{sec:SJthm}
In this section, we prove the existence 
of martingale solutions in the sense of 
Definition \ref{def:mart_sol}. We do so in three steps. 

First, taking the convergence in law  
proven for $\{R_N\}$, $\{S_N\}$, 
in Proposition \ref{thm:tightness_N} and 
for $\{u_N\}$ in Lemma \ref{thm:tilde_u}, 
we apply to them the Skorokhod--Jakubowski representation 
theorem. This will produce new variables, the 
Skorokhod representatives $\{\tilde{X}_N 
:= (\tilde{R}_N, \tilde{S}_N, \tilde{u}_N)\}$,  
on a new probability space, that are 
equal in law to $\{X_N := (R_N, S_N, u_N)\}$, 
but $\tilde{X}_N$ converges almost surely. 
(We wrote $X_N$ and $\tilde{X}_N$ to fix 
ideas, but shall need to expand their 
definitions later for technical reasons.)

Dudley maps $\mathcal{T}_N$, which were first 
proposed by Dudley \cite{Dud1985},  map 
the new probability space to the original 
one in a measure-preserving way. They give us a 
way to write $\tilde{X}_N = X_N \circ \mathcal{T}_N$. 
Using the maps $\mathcal{T}_N$, we are 
able very readily to derive equations 
for the new variables $\tilde{R}_N$ and $\tilde{S}_N$. 
They also clarify the mechanics of the 
joint equality of laws arising from the 
Skorokhod--Jakubowski theorem, which 
play an important role in transferring  
properties of $X_N$ onto $\tilde{X}_N$. 

Finally, we use a convergence theorem 
\cite[Lemma 2.1]{Debussche:2011aa} for stochastic integrals to take 
limits of the equations for $\tilde{R}_N$ and $\tilde{S}_N$ 
to limiting equations by showing that 
the stochastic integral converges a.s.~strongly 
in $L^2([0,T])$. 
This then allows us to 
conclude that the limits of  $\tilde{R}_N$ and $\tilde{S}_N$ 
are in fact martingale solutions to \eqref{eq:vvw}.

\subsection{Skorokhod representatives}
The sole purpose of this short subsection is to 
construct Skorokhod--Jakubowski representatives 
of $R_N$, $S_N$ and related variables that 
converge a.s.~ on a new probability space, 
and exhibit the Dudley maps related to the 
the representations. Define the path spaces:
\begin{equation}\label{eq:pathspaces_defin}
\begin{aligned}
&\mathcal{Y}_{R1}  = \mathcal{Y}_{S1} 
:=  L^2([0,T]\times \T), \quad 
\mathcal{Y}_{R2}  = \mathcal{Y}_{S2} 
:=  C([0,T];L^2(\T)-w),\\
&\mathcal{Y}_u := C([0,T]\times \T), \quad 
\mathcal{Y}_{R0} = \mathcal{Y}_{S0} 
:= L^2(\T),\quad \mathcal{Y}_W := C([0,T]).
\end{aligned}
\end{equation}

Let $\mathcal{Y} := 
\mathcal{Y}_{R1}  \times \mathcal{Y}_{S1} \times
\mathcal{Y}_{R2}  \times \mathcal{Y}_{S2} \times
 \mathcal{Y}_u\times 
\mathcal{Y}_{R0}  \times \mathcal{Y}_{S0} \times
\mathcal{Y}_W$.

The spaces $\mathcal{Y}_\cdot$ are individually 
quasi-Polish spaces, in the sense that 
there exist a countable, point-separating collection 
of continuous maps $\{f_j : \mathcal{Y}_\cdot \to [0,1]\}_{j \in \N}$. 
A countable product of quasi-Polish spaces remains a
quasi-Polish space. The one-to-one continuous 
injection into the Hilbert cube (which is Polish), 
is the crucial property identified by Jakubowski \cite{Jakubowski:1997aa} 
under which an extension of the classical 
Skorokhod representation theorem holds. 
(We refer to \cite{Jakubowski:1997aa}, 
\cite[Section 3.3]{Brzezniak:2016wz}, and 
\cite[Appendix B]{GHKP2022},  for a more 
extensive introduction to quasi-Polish spaces.) 
We use the Skorokhod--Jakubowski theorem to impose 
the topology of pointwise convergence in the 
probability variable, allowing us to leverage 
the deterministic compactness results in 
Section \ref{sec:galerkin_compactness} to handle 
convergence of the Galerkin approximation 
\eqref{eq:galerkin_lim} in the remaining 
spatio-temporal variables.

\begin{prop}[Skorokhod--Jakubowksi theorem]\label{thm:skorokhod_N}
Consider the solutions to \eqref{eq:galerkin_lim} 
given by Proposition \ref{thm:galerkin_wp} for each $N \in \N$. 
There exist:
\begin{itemize}
\item[(i)] a probability space $\tilde{E}:= (\tilde{\Omega}, 
	\tilde{\mathcal{F}}, \tilde{\mathbb{P}})$, 
\item[(ii)] $\mathcal{Y}$-valued random variables defined on $\tilde{E}$, 
	$$
	\tilde{X}_n := (\tilde{R}_n, \tilde{S}_n, \tilde{\xi}_n, \tilde{\zeta}_n, \tilde{u}_n, 
		\tilde{R}_n^0, \tilde{S}_n^0, \tilde{W}_n) 
	\,\,\,\, \text{
	and }\,\,\,\, 
	\tilde{X} := (\tilde{R}, \tilde{S}, \tilde{\xi}, \tilde{\zeta},\tilde{u}, 
		\tilde{R}^0, \tilde{S}^0, \tilde{W}),
	$$
	
\item[(iii)] a subsequence 
	$\{N_n\}_{n \in \N}\subseteq \N$ such that 
	the joint equality of laws hold:
\begin{align*}
&\tilde{X}_n \sim X_{N_n} 
	:= ({R}_{N_n}, {S}_{N_n}, {R}_{N_n}, {S}_{N_n}, 
		u_{N_n}, {\bf P}_{N_n}{R}^0, 
		{\bf P}_{N_n}{S}^0, W),\\
	\intertext{and}
&\tilde{X}_n \to \tilde{X}\quad
	\text{ in $\mathcal{Y}$, $\tilde{\mathbb{P}}$-a.s., and }
\end{align*}
\item[(iv)] for each finite $n$, 
maps $\mathcal{T}_n: 
	\tilde{\Omega} \to \Omega$ (Dudley maps) which preserve 
measure in the sense that 
$\bk{\mathcal{T}_n}_* \circ  \tilde{\mathbb{P}}
	= \mathbb{P}$, such that 
$$
\tilde{X}_n  = X_{N_n} \circ \mathcal{T}_n.
$$
\end{itemize}
\end{prop} 

Property (iv) implies the joint 
equality of laws in (iii), but we 
spell this out for clarity.

\begin{proof}
Apart from the existence of measure-preserving Dudley maps 
$\mathcal{T}_n$, the proposition statement 
follows from the Skorokhod--Jakubowski 
theorem \cite[Theorem 2]{Jakubowski:1997aa} 
once we show that the laws of 
$\{X_N\}$ are tight in $\mathcal{Y}$. This tightness  
in turn follows form the tightness of laws of 
the elements of $X_N$ in the corresponding 
factors of $\mathcal{Y}$. The expanded version 
of the Skorokhod--Jakubowski statement which 
includes the existence of Dudley maps can 
be found in \cite[Theorem A.1]{Punshon-Smith:2018aa} 
(and is the quasi-Polish extension of 
\cite[Theorem 1.10.4]{VW1996}).

The respective tightness of the laws of 
$\{R_N\}_{N \in \N}$ and $\{S_N\}_{N \in \N}$ 
on $\mathcal{Y}_{R_i}$ and $\mathcal{Y}_{S_i}$, $i = 1,2$, 
follow from Proposition \ref{thm:tightness_N}.

The tightness of $\{u_N\}$ 
on $\mathcal{Y}_u$ is the result of 
Lemma \ref{thm:tilde_u}.

The tightness of 
$\{{\bf P}_NR^0\}_{N \in \N}$ 
and $\{{\bf P}_NS^0\}_{N \in \N}$ 
on $\mathcal{Y}_{R0} $ and 
$\mathcal{Y}_{S0}$, respectively, follow 
from the property of the projection operator.

Finally, the tightness of the law of the  
Brownian motion $W$ restricted to $[0,T]$ in 
$C([0,T])$ is standard.
\end{proof}

Define the filtrations:
\begin{align*}
\{\tilde{\mathcal{F}}^n_t\}_{t \in [0,T]} 
= {\Sigma}\bk{\Sigma\bk{\tilde{X}_n|_{[0,t]}} 
	\cup \Sigma\bk{\{N : \tilde{\mathbb{P}}(N) = 0\}}}.
\end{align*}
The limit filtration $\{\tilde{\mathcal{F}}_t\}_{t \in [0,T]}$ 
is constructed similarly.
The following 
elementary result holds by a standard argument 
(see, e.g., \cite[Lemma 9.10]{Brzezniak:2011aa}):
\begin{lem}\label{thm:W_convg}
Let $\tilde{W}_n$ and $\tilde{W}$ be as 
constructed in Proposition \ref{thm:skorokhod_N}. 
For each $n$, $\tilde{W}_n$ is a standard 
$\{\tilde{\mathcal{F}}^n_t\}_{t \in [0,T]}$-Brownian motion, 
and similarly for $\tilde{W}$ in the limit. 
\end{lem}

\subsection{Consequences of the joint equality of laws}

In this subsection, we use the Dudley maps 
$\mathcal{T}_n$ to establish several bounds 
on $\tilde{R}_n$ and $\tilde{S}_n$ that are 
direct consequences of the equality of laws. 
The Dudley maps will also simplify the identification 
of the equations satisfied by $\tilde{R}_n$ and 
$\tilde{S}_n$. 

We begin by identifying $\tilde{\xi}_n$ with $\tilde{R}_n$, etc.:
\begin{lem}\label{thm:xiR_zetaS}
Let $\tilde{\xi}$, $\tilde{\zeta}$, $\tilde{R}$, $\tilde{S}$, 
and for each $n$, $\tilde{\xi}_n$, $\tilde{\zeta}_n$, 
$\tilde{R}_n$ and $\tilde{S}_n$ be as constructed 
as in Proposition \ref{thm:skorokhod_N}. We have the 
identifications:
\begin{align*}
\tilde{\xi}_n = \tilde{R}_n, \quad 
\tilde{\zeta}_n = \tilde{S}_n, \quad 
\tilde{\xi} = \tilde{R}, \quad 
\tilde{\zeta} = \tilde{S}. 
\end{align*}
\end{lem}

\begin{proof}
Using the Dudley maps $\mathcal{T}_n$, 
we immediately find:
\begin{align*}
\tilde{\xi}_n = R_{N_n} \circ \mathcal{T}_n = \tilde{R}_n, \quad
\tilde{\zeta}_n = S_{N_n} \circ \mathcal{T}_n = \tilde{S}_n.
\end{align*}

Suppose $\{f_n\}$ is a sequence such that 
$f_n \to F_1$ in $\mathcal{Y}_{R1} =   L^2([0,T]\times \T)$ and 
$f_n \to F_2$ in $\mathcal{Y}_{R2} = C([0,T];L^2(\T)-w) $. 
Then for any $g \in  L^2([0,T])$, $h \in L^2(\T)$, 
the strong convergence $f \to F_1$ 
implies 
$\int_0^T \int_\T f_n gh \,\d x \,\d t 
\to \int_0^T \int_\T F_1 gh \,\d x \,\d t$.
On the other hand, by the convergence 
$f_n \to F_2$ in $\mathcal{Y}_{R2} $, 
$\int_\T f_n h \,\d x \to  \int_\T F_2 h\,\d x$ 
in $C([0,T])$, from which it follows 
that 
$\int_0^T \int_\T f_n gh \,\d x\,\d t 
\to  \int_0^T\int_\T F_2 gh\,\d x\,\d t$. 
The arbitrariness of $g$ and $h$ 
implies that $F_1 = F_2$, $(t,x)$-a.e. 
Applying this to $f_n = \tilde{R}_n(\tilde{\omega})$ and 
$f_n = \tilde{S}_n(\tilde{\omega})$ for 
each $\tilde{\omega} \in \tilde{\Omega}$, 
the lemma follows.
\end{proof}
It would have been possible to identify 
compact sets in the supremum topology of 
$\mathcal{Y}_R := \mathcal{Y}_{R1} \cap \mathcal{Y}_{R2}$ 
and taken $\mathcal{Y}_R$ for a path space, following 
the Dubinsky theorem of \cite[Lemmas 3.1, 3.3]{Brzezniak:2013aa} 
(and references included there). We have chosen 
to take two copies $R_{N_n}$ on different path spaces, 
and identify their Skorokhod representatives 
$\tilde{R}_n$ and $\tilde{\xi}_n$ afterwards. 
This is a more flexible approach and avoids the 
need for establishing further compactness theorems. 
For the remainder of this section, 
we shall not refer to $\tilde{\xi}_n$, $\xi$, 
$\tilde{\zeta}_n$, and $\tilde{\zeta}$ any longer.

We next identify $\tilde{u}_n$ as a 
function of $\tilde{R}_n$ and $\tilde{S}_n$. 
\begin{lem}\label{thm:constitutive_n_SJ}
With $\mathfrak{u}$ as constructed 
in \eqref{eq:u_construct}, we have
\begin{align*}
\tilde{u}_n = \mathfrak{u}_{N_n}(\tilde{R}_n,\tilde{S}_n), \qquad 
2c_{N_n}(\tilde{u}_n) \pd_x \tilde{u}_n = \tilde{R}_n - \tilde{S}_n.
\end{align*}
\end{lem}

\begin{proof}
Recall the Lipschitz bijection $F$ 
defined in \eqref{eq:F_defin} and the inverse 
operator $\pd_x^{-1}$, which is continuous 
$H^s_0(\T) \to H^{s + 1}_0(\T)$ for any $s \in \R$, 
defined in \eqref{eq:dx_inverse}. 
From the definitions of 
$\mathcal{Y}_\cdot$, the following map is continuous: 
\begin{align*}
\mathcal{Y}_R \times \mathcal{Y}_S  \ni
(\tilde{R}_n, \tilde{S}_n) 
\mapsto 
F^{-1}_{N_n} (\pd_x^{-1} \bk{\tilde{R}_n - \tilde{S}_n} ) \in \mathcal{Y}_u.
\end{align*}
Therefore, the map is continuous (and hence measurable)
on the sets where the values of $\tilde{R}_n$ 
and $\tilde{S}_n$ are uniformly bounded.
Using the Dudley maps given in Proposition 
\ref{thm:skorokhod_N}, 
we can then conclude that 
$$
\tilde{u}_n = u_{N_n} \circ \mathcal{T}_n 
	= \mathfrak{u}_{N_n}(R_{N_n}\circ \mathcal{T}_n, 
		S_{N_n}\circ \mathcal{T}_n)
	= \mathfrak{u}_{N_n}(\tilde{R}_n, \tilde{S}_n).
$$
The final equality follows from the definition 
of $F$ used in $\mathfrak{u}$ (see \eqref{eq:RNSN_dFdx}). 
\end{proof}
 
Additionally, using the joint equality of laws 
(or directly by Dudley maps), we have the 
following uniform bound from \eqref{eq:galerkin_L2}:
\begin{lem}\label{thm:energy_N_SJ}
Let $\tilde{R}_n$, $\tilde{S}_n$ be as constructed 
in Proposition \ref{thm:skorokhod_N}. The following 
bound holds:
\begin{align*}
\tilde{\Ex} \sup_{t \in [0,T]}
	\bk{ \int_{\T} \tilde{R}_n^2 + \tilde{S}_n^2 \,\d x}^{p_0}
+ \nu^{p_0} \tilde{\Ex} \bk{\int_0^T \int_{\T}
	\abs{ \pd_x \tilde{R}_n}^2 
		+ \abs{\pd_x \tilde{S}_n}^2\,\d x \,\d t }^{p_0}
\lesssim_{\sigma, T} 1.
\end{align*}

And we have the following bound in the limit: 
\begin{equation}\label{eq:energy_SJ}
\begin{aligned}
\tilde{\Ex} \sup_{t \in [0,T]}&
	\bk{ \int_{\T} \tilde{R}^2 
	+ \tilde{S}^2 \,\d x}^{p_0}\\
&+\nu^{p_0} \tilde{\Ex}\bk{ \int_0^T \int_{\T} 
	\abs{\pd_x\tilde{R}}^2 + \abs{\pd_x \tilde{S}}^2
	 \,\d x \,\d t}^{2p_0}\lesssim_{\sigma,T} 1.
\end{aligned}
\end{equation}

\end{lem}

\begin{proof}
The first uniform bound holds as 
$\tilde{R}_n = R_{N_n} \circ \mathcal{T}_n$, 
and hence takes values in $L^\infty([0,T];L^2(\T))$. 
Since taking the norm is a continuous function, 
the map 
$\tilde{\omega} \mapsto 
	\norm{R_{N_n} \circ \mathcal{T}_n}_{L^\infty_t L^2_x}^2$ 
is measurable. 
Therefore, it is possible to effect a change-of measure
\begin{equation}\label{eq:uniform_tildeR}
\begin{aligned}
\tilde{\Ex} \norm{\tilde{R}_n}_{L^\infty_t L^2_x}^{2p_0}
&= \int_{\tilde{\Omega}} 
	\norm{R_{N_n} \circ \mathcal{T}_n}_{L^\infty_t L^2_x}^{2p_0}
	\tilde{\mathbb{P}}(\d \tilde{\omega})\\
& = \int_{{\Omega}} 
	\norm{R_{N_n}}_{L^\infty_t L^2_x}^{2p_0}
	\bk{\mathcal{T}_n}_* \circ \tilde{\mathbb{P}}(\d {\omega})
= \Ex \norm{R_{N_n}}_{L^\infty_t L^2_x}^{2p_0} 
\overset{\eqref{eq:galerkin_L2}}{\lesssim} 1
\end{aligned}
\end{equation}
We can argue similarly for $\tilde{S}_n$, 
as well as repeat the argument on 
$L^2([0,T];H^1(\T))$ for the derivatives. 

In the limit, adapting the reasoning in the proof of 
Lemma \ref{thm:xiR_zetaS}, since $\tilde{R}_n$ is bounded in 
in $L^{2p_0}(\tilde{\Omega};L^\infty([0,T];L^2(\T)))$, 
it has a weak* limit $\overline{R}$ in this 
space by the Banach--Alaoglu theorem. 
Therefore, for any $Y \in L^\infty_{\omega,t,x}$, 
$$
\tilde{\Ex} \int_0^T \int_\T Y \tilde{R}_n \,\d x \,\d t 
\to \tilde{\Ex} \int_0^T \int_\T Y \overline{R} \,\d x \,\d t.
$$
On the other hand, $\tilde{R}_n\to \tilde{R}$ 
in $L^2([0,T] \times \T)$, $\tilde{\mathbb{P}}$-a.s., 
from Proposition \ref{thm:skorokhod_N}, and 
by \eqref{eq:uniform_tildeR}, we have the 
uniform bound $\tilde{\Ex} \norm{\tilde{R}_n}_{L^2_{t,x}}^2 \lesssim 1$. 
Pointwise convergence and a uniform bound 
implies the weak convergence 
$\tilde{R}_n \rightharpoonup \tilde{R}$ 
in $L^2(\tilde{\Omega}\times [0,T] \times \T)$. 
Hence for any $Y \in L^\infty_{\omega,t,x}$, 
$$
\tilde{\Ex} \int_0^T \int_\T Y \tilde{R}_n \,\d x \,\d t 
\to \tilde{\Ex} \int_0^T \int_\T Y \tilde{R} \,\d x \,\d t.
$$
Therefore $\tilde{R} = \overline{R}$ 
$(\tilde{\omega},t,x)$-a.e., and $\tilde{R}_n 
	\xrightharpoonup{*} \tilde{R}$ 
in $L^{2p_0}(\tilde{\Omega};L^\infty([0,T];L^2(\T)))$ 
and $\tilde{R}$ is included in that space.

Again, similar arguments can be 
made for $\tilde{S}$ and the derivatives. 

\end{proof}

\begin{rem}\label{rem:tildeu_n_uniform}
The uniform bound of Lemma \ref{thm:energy_N_SJ}, 
along with Lemma \ref{thm:constitutive_n_SJ} 
and the boundedness of $c$ from 
below imply that $\tilde{u}_n \in_b 
\mathcal{X}_2 := L^{2p_0}(\tilde{\Omega};
	L^\infty([0,T]; H^1(\T)) \cap L^2([0,T];H^2(\T)))$. 
The bound \eqref{eq:u_highermome} also implies 
$\tilde{u}_n \in_b L^{2p_0}(\tilde{\Omega};C([0,T];H^\beta(\T))$ 
for any $\beta < 1$. 
Using a uniqueness of weak limits 
argument as used in the proof of Lemma 
\ref{eq:energy_SJ} above, it can be shown that 
$\tilde{u}$ is also bounded in $\mathcal{X}_2$. 
\end{rem}

With the help of Dudley maps, we now prove that 
$\tilde{R}_n$ satisfies the equation on the new probability 
space $\tilde{E}$  that $R_{N_n}$ 
satisfies on $(\Omega, \mathcal{F}, \mathbb{P})$.
The same argument will yield the corresponding 
equation for  $\tilde{S}_n$. 
Fixing $\varphi \in C^2(\T)$, we define as 
usual the following quantities for convenience, 
which is the equation in weak form, less the martingale terms:
\begin{align}
\notag
&\begin{aligned}
M_n(t) &:= \int_{\T} R_{N_n}(t) \varphi\,\d x 
	- \int_{\T} R^0_{N_n}\varphi\,\d x
	+ \nu\int_0^t \int_{\T} \pd_x \varphi 
		\,\pd_x \bk{R_{N_n} + S_{N_n}}\,\d x\,\d t'\\
	&\qquad + \int_0^t \int_{\T} \pd_x{\bf P}_{N_n} \varphi 
			c_{N_n}(u_{N_n}) R_{N_n}   \,\d x\,\d t' \\
	&\qquad +\int_0^t \int_{\T}  {\bf P}_{N_n} \varphi 
			\tilde{c}_{N_n}(u_{N_n}) 
			\big( R_{N_n} - S_{N_n} \big)^2\,\d x \,\d t'\\
&\qquad  + \frac12 \int_0^t \int_{\T} \sigma\,\pd_x\bk{\sigma\, 
			 {\bf P}_{N_n} \varphi} \pd_x \bk{R_{N_n} + S_{N_n}}\,\d x \,\d t',
\end{aligned}\\
\label{eq:N_martingale2}
&\begin{aligned}
\tilde{M}_n(t) &:= \int_{\T} \tilde{R}_{n}(t) \varphi\,\d x 
		- \int_{\T} \tilde{R}^0_n\varphi\,\d x
		+ \nu\int_0^t \int_{\T} \pd_x \varphi 
		\,\pd_x \bk{\tilde{R}_{n} + \tilde{S}_n}\,\d x \,\d t'\\
	&\qquad + \int_0^t \int_{\T} \pd_x  {\bf P}_{N_n} 
		\varphi c_{N_n}(\tilde{u}_{n}) \tilde{R}_{n}  \,\d x\,\d t' \\
	&\qquad + \int_0^t \int_{\T}  {\bf P}_{N_n} \varphi \tilde{c}_{N_n}(\tilde{u}_{n}) 
			\big( \tilde{R}_{n} - \tilde{S}_{n} \big)^2\,\d x \,\d t'\\
&\qquad  + \frac12 \int_0^t \int_{\T}\sigma\,\pd_x\bk{\sigma\, 
			 {\bf P}_{N_n} \varphi} \pd_x\bk{\tilde{R}_{n} + \tilde{S}_n} \,\d x\,\d t'.
\end{aligned}
\end{align}

\begin{prop}[Approximating equation on 
	the new probability space]\label{thm:nth_eq_Nlimit}
The following equation holds for each $n$, 
$\tilde{\mathbb{P}}$-a.s. for every $t \in [0,T]$:
\begin{align}\label{eq:nth_eq_Nlimit}
\tilde{M}_n(t) = -\int_0^t  \int_{\T}  \pd_x \bk{{\bf P}_{N_n}
	\varphi  \,\sigma}\bk{\tilde{R}_{n} 
	+ \tilde{S}_{n}}\,\d x\,\d \tilde{W}_n.
\end{align}
\end{prop}

\begin{rem}
By the equality of laws and the construction 
of $R_{N_n}$ in Section \ref{sec:scheme}, 
the variable $\tilde{R}_n$ in fact takes 
values in $H^3(\T)$. Therefore 
\eqref{eq:nth_eq_Nlimit} also holds strongly. 
\end{rem}

\begin{proof}
This statement concerning the validity of 
the $n$th equation on the new probability 
space $\tilde{E}$ is often established 
via a martingale identification argument. 
Instead, we exploit the Dudley maps. 

From Proposition \ref{thm:skorokhod_N}, 
and the a.s.~continuity of each term in 
$M_{N_n}$ as a map of $X_{N_n}$, we have 
the first equality of:
$$
\tilde{M}_n (t) = M_{N_n}\circ \mathcal{T}_n(t)
=  -\bk{\int_0^t  \int_{\T}  \pd_x \bk{{\bf P}_{N_n}
	\varphi  \,\sigma}{R}_{N_n}\,\d x\,\d {W}}\circ \mathcal{T}_n.
$$
The second equality follows by construction. 
It remains to show that the right hand side 
is equal to $ -\int_0^t  \int_{\T}  \pd_x \bk{{\bf P}_{N_n}
	\varphi  \,\sigma}\tilde{R}_{n}\,\d x\,\d \tilde{W}_n$.
This follows from an approximation argument 
such as found in \cite[Section 4.3.4]{Bensoussan:1995aa}. 
We summarise this argument here. 

Let $J_\ep(t) = \frac1\ep e^{-t/\ep}$. 
Consider the mollified integrand 
\begin{align*}
Y_{N_n}^\ep(t') := \int_0^{t'} \tilde{J}_\ep(t' - s)
			\int_{\T}  \pd_x \bk{{\bf P}_{N_n}
			\varphi  \,\sigma}{R}_{N_n}(s)\,\d x\,\d s.
\end{align*}
Then $Y_{N_n}^\ep \to Y_{N_n}^0$ in $L^2(\Omega \times [0,T])$, 
so the corresponding integrals $\int_0^t Y_{N_n}^\ep \,\d W$ 
converges to $\int_0^t Y_{N_n}^0 \,\d W$ 
in $L^2(\Omega \times [0,T])$ by the It\^o isometry.
On the other hand, by the temporal regularity 
of the mollified $Y_{N_n}^\ep$, which a.s.~has no 
quadratic variation, 
\begin{align*}
\bk{\int_0^t Y_{N_n}^\ep \,\d W }\circ \mathcal{T}_n
&= \bk{Y_{N_n}^\ep(t) W(t)}\circ \mathcal{T}_n 
	- \bk{\int_0^t \pd_s Y_{N_n}^\ep  W(s) \,\d s}\circ \mathcal{T}_n \\
&= \int_0^t \tilde{J}_\ep(t - s)\int_{\T}  \pd_x \bk{{\bf P}_{N_n}
	\varphi  \,\sigma}\tilde{R}_n(s)\,\d x \,\d s \tilde{W}_n(t) \\
&\qquad	-   \int_0^t \int_0^{t'}\pd_s J_\ep(t' - s)
	\int_{\T}  \pd_x \bk{{\bf P}_{N_n}
	\varphi  \,\sigma}\tilde{R}_n\,\d x\,\d s\, \tilde{W}_n(t')\,\d t'\\
& = -\int_0^t  \int_0^{t'} J_\ep(t' - s)\int_{\T}  \pd_x \bk{{\bf P}_{N_n}
	\varphi  \,\sigma}\tilde{R}_{n}\,\d x\,\d s\,\d \tilde{W}_n\\
& \xrightarrow{\ep \downarrow 0}
 -\int_0^t  \int_{\T}  \pd_x \bk{{\bf P}_{N_n}
	\varphi  \,\sigma}\tilde{R}_{n}\,\d x\,\d \tilde{W}_n, \quad 
	\text{in $L^2(\Omega \times [0,T])$}.
\end{align*}

This establishes the equation for $\tilde{R}_n$ 
and proves the proposition.
\end{proof}

\subsection{Martingale solutions to the 
	viscous variational wave equations}\label{sec:mart_sol}
In this subsection, we shall take limits to establish 
an equation for $(\tilde{R}, \tilde{S})$. 
Already having limiting results for $\tilde{R}_n$ 
and $\tilde{S}_n$ from Proposition \ref{thm:skorokhod_N}, 
we first establish limit theorems for $\tilde{u}_n$ 
and $c_{N_n}(\tilde{u}_n)$. In this task we 
again bring to mind the definition 
\eqref{eq:F_defin} of $F$, the anti-derivative of $c$.

We begin with a simple observation:
\begin{lem}\label{thm:Fu_n-Fu}
Let $\tilde{u}_n$, $\tilde{u}$ be as constructed in 
Proposition \ref{thm:skorokhod_N}. The following 
convergence holds:
$$
\Ex \norm{F(\tilde{u}_n) 
	- F(\tilde{u})}_{C([0,T] \times \T)}^2 \to 0.
$$
\end{lem}

\begin{proof}

By the Lipschitz continuity of $F$, 
\begin{align*}
\norm{F(\tilde{u}_n) 
	- F(\tilde{u})}_{L^2_{\tilde{\omega}} C_{t,x}} 
	\lesssim_\kappa
\norm{\tilde{u}_n 
	- \tilde{u}}_{L^2_{\tilde{\omega}} C_{t,x}}.
\end{align*}
Given the $\tilde{\mathbb{P}}$-a.s.~convergence 
$\tilde{u}_n \to \tilde{u}$ in $C([0,T] \times \T)$ 
(Proposition \ref{thm:skorokhod_N}), we require 
only a higher moment bound on 
$\norm{\tilde{u}_n 
	- \tilde{u}}_{C_{t,x}}$ to conclude. 
	
By \eqref{eq:u_highermome} and 
arguing as in Lemma \ref{thm:energy_N_SJ} 
(cf. Remark \ref{rem:tildeu_n_uniform}), 
we have the bounded inclusion 
$\{\tilde{u}_n\} \subset_b 
L^{2p_0}(\tilde{\Omega};C([0,T]\times \T)) $. 
Therefore, Vitali's convergence theorem implies 
that $\norm{\tilde{u}_n 
	- \tilde{u}}_{L^2_{\tilde{\omega}} C_{t,x}} \to 0$.
\end{proof}

Our key lemma concerning $c_{N_n}(\tilde{u}_n)$ 
is the following:
\begin{lem}\label{thm:c_convergences_N}
For any $m> 0$, 
\begin{align}\label{eq:c(u_n)_converge}
c_{N_n}(\tilde{u}_n) \to  c(\tilde{u}), 
\qquad \tilde{c}_{N_n}(\tilde{u}_n) \to 
	 \tilde{c}(\tilde{u}), 
\end{align}
$\tilde{\mathbb{P}}$-a.s.~in $C([0,T]\times \T)$.

Moreover the following equations hold $\tilde{\mathbb{P}}
	\otimes \d x\otimes \d t$-a.e.: 
\begin{align}\label{eq:c_x_tilde_c_SJ}
\tilde{u} = \mathfrak{u}(\tilde{R}, 
	\tilde{S}), \qquad
2\pd_x F(\tilde{u}) 
	= {\tilde{R} - \tilde{S}}.
\end{align}

\end{lem}

\begin{proof}
From Proposition \ref{thm:skorokhod_N} 
$\tilde{u}_n \to \tilde{u}$, $\tilde{\mathbb{P}}$-a.s.~in 
$C([0,T] \times \T)$.  The 
$\tilde{\mathbb{P}}$-a.s.~convergences follow 
from the convergence $c_{N_n} \to c$ in $C(\R)$  
provided by \eqref{eq:c_N_approx}:
\begin{align*}
c_{N_n}(\tilde{u}_n) - c(\tilde{u}) 
&= c_{N_n}(\tilde{u}_n) - c(\tilde{u}_n) + c(\tilde{u}_n) - c(\tilde{u})\\
& \le \norm{c_{N_n} -c}_{C(\R)} + \bk{c(\tilde{u}_n) - c(\tilde{u})}.
\end{align*}
The convergence of $\tilde{u}_n$ implies that 
for a.e. $\tilde{\omega} \in \tilde{\Omega}$, 
$\{\tilde{u}_n(\tilde{\omega},t,x): n \in \N,\, (t,x) \in [0,T] \times \T\}$ 
take values 
on a compact set $K(\tilde{\omega}) \subset \R$.  
Even if $c$ is only continuous, it is uniformly 
continuous on $K$ with some modulus $\varpi$.
Therefore, $\tilde{\mathbb{P}}$-a.s.,
$$
\norm{c_{N_n}(\tilde{u}_n) - c(\tilde{u})}_{C_{t,x}} 
\le o_{n \uparrow \infty} (1)
	+ \frac{ \bk{c(\tilde{u}_n) - c(\tilde{u})}}
		{\varpi\bk{\abs{\tilde{u}_n - \tilde{u}}}} 
	\varpi(\norm{\tilde{u}_n - \tilde{u}}_{C_{t,x}}) 
= o_{n \uparrow \infty} (1).
$$
Since we did not use the Lipschitz continuity 
of $c$, the same argument holds for $\tilde{c}_{N_n}$ 
and $\tilde{c}$ in place of $c_{N_n}$ and $c$.

It remains then to prove the first equality of 
\eqref{eq:c_x_tilde_c_SJ}, from which its second 
equality follows by the definition of $F$ 
in \eqref{eq:F_defin} and the bounds 
of Remark \ref{rem:tildeu_n_uniform}. 

From Lemma \ref{thm:constitutive_n_SJ}, 
$\mathfrak{u}_{N_n}(\tilde{R}_n,\tilde{S}_n) = \tilde{u}_n$. 
In view of Lemma \ref{thm:Fu_n-Fu}, which implies 
the convergence $F(\tilde{u}_n) \to F(\tilde{u})$ 
in $L^1(\tilde{\Omega}\times [0,T]\times \T)$, 
we shall show that 
\begin{align}\label{eq:u_frak_conv_L1}
\Ex \norm{F(\mathfrak{u}_{N_n}(\tilde{R}_n, \tilde{S}_n) )
	- F(\mathfrak{u}(\tilde{R},\tilde{S}))}_{L^1_{t,x}} 
	\xrightarrow{n\uparrow \infty} 0.
\end{align}
Together, these facts imply 
$$
\Ex \norm{F(\tilde{u})
	- F(\mathfrak{u}(\tilde{R},\tilde{S}))}_{L^1_{t,x}}  = 0.
$$
By the invertibility of $F$, 
the first equality of \eqref{eq:c_x_tilde_c_SJ} ensues.

From the linearity of the inverse differentiation 
$\pd_x^{-1}$, the ``constitutive relation'' for $\tilde{u}_n$ 
in Lemma \ref{thm:constitutive_n_SJ}, and from the 
construction \eqref{eq:u_construct}, we have
\begin{equation}\label{eq:F_difference_ulim}
\begin{aligned}
2F(\mathfrak{u}_{N_n}(\tilde{R}_n, \tilde{S}_n) )
	- 2F(\mathfrak{u}(\tilde{R},\tilde{S}))
& = \pd^{-1}_x\bk{ \bk{\frac{c(\tilde{u}_n)}
	{c_{N_n}(\tilde{u}_n)}  - 1} 
	\bk{\tilde{R}_n - \tilde{S}_n}}\\
&\quad\,\, +  {\pd_x^{-1}\bk{\tilde{R}_n - \tilde{R}}} 
-{\pd_x^{-1} \bk{\tilde{S}_n - \tilde{S}}}. 
\end{aligned}
\end{equation}

Using the convergence \eqref{eq:c(u_n)_converge}, 
\begin{align*}
\abs{\frac{c(\tilde{u}_n)}{c_{N_n}(\tilde{u}_n)}  - 1}
\le \kappa^3 \abs{c_{N_n}(\tilde{u}_n) - c(\tilde{u}) 
	+ c(\tilde{u}) - c(\tilde{u}_n)} = o_{n\uparrow \infty}(1)
\end{align*}
in $L^2(\tilde{\Omega};C([0,T]\times\T))$. 
And therefore,
\begin{align*}
&\norm{\pd^{-1}_x\bk{ \bk{\frac{c(\tilde{u}_n)}
	{c_{N_n}(\tilde{u}_n)}  - 1} 
	\bk{\tilde{R}_n - \tilde{S}_n}}}_{L^1_{\tilde{\omega},t,x}}^2\\
&\lesssim o_{n\uparrow \infty}(1) 
	\Ex \bk{\int_0^T \int_\T \abs{\int_0^x \bk{\tilde{R}_n - \tilde{S}_n}\,\d y 
	- x \int_\T {\tilde{R}_n - \tilde{S}_n}\,\d y }\,\d x\,\d t}^2\\
&\xrightarrow{n\uparrow \infty} 0,
\end{align*}
by the uniform-in-$n$ energy bound of 
Lemma \ref{thm:energy_N_SJ}.

Using the definition  \eqref{eq:dx_inverse} of 
$\pd_x^{-1}$, 
\begin{align*}
&\norm{\pd_x^{-1}\bk{\tilde{R}_n - \tilde{R}}}_{C_{t,x}}  \\
&\le \sup_{t,x}\abs{\int_0^x 
	\frac{\tilde{R}_n(t,y) - \tilde{R}(t,y)}2 \,\d y
 - x \int_\T \frac{\tilde{R}_n(t,y) - \tilde{R}(t,y)}2 \,\d y}\\
&\quad\,\,  + \sup_{t}\abs{\int_\T \int_0^y 
	\frac{\tilde{R}_n(t,z) - \tilde{R}(t,z)}2 \,\d z
- y\int_\T \frac{\tilde{R}_n(t,z) - \tilde{R}(t,z)}2 \,\d z}\\
& \xrightarrow{n\uparrow \infty} 0, 
\quad \text{$\tilde{\mathbb{P}}$-a.s.},
\end{align*}
because $\tilde{R}_n \to \tilde{R}$ in $C([0,T];L^2(\T)-w)$, 
$\tilde{\mathbb{P}}$-a.s.
Again using the energy bound of 
Lemma \ref{thm:energy_N_SJ}, 
Vitali's convergence theorem implies 
\begin{align*}
\norm{\pd_x^{-1}\bk{\tilde{R}_n - \tilde{R}}}_{
	L^1_{\tilde{\omega},t,x}}  \le
\norm{\pd_x^{-1}\bk{\tilde{R}_n - \tilde{R}}}_{
	L^1_{\tilde{\omega}}C_{t,x}}  
\xrightarrow{n\uparrow \infty} 0.
\end{align*}
The corresponding difference  
$\big\|\pd_x^{-1} \big(\tilde{S}_n - \tilde{S}\big)\big\|_{
	L^1_{\tilde{\omega},t,x}}$ 
for $S$ vanishes similarly. 
This establishes \eqref{eq:u_frak_conv_L1}.

\end{proof}

We can finally prove the convergence of the equation 
\eqref{eq:nth_eq_Nlimit}. 
As in \eqref{eq:N_martingale2}, 
let us define the limiting quantities 
\begin{equation}
\label{eq:N_martingale3}
\begin{aligned}
\tilde{M}(t) &:= \sum_{i = 1}^5 I_i,
\end{aligned}
\end{equation}
where
\begin{align*}
I_1 &:= \int_{\T} \tilde{R}(t) \varphi\,\d x
	- \int_{\T} \tilde{R}^{0}\varphi\,\d x,\qquad
I_2 :=  \nu\int_0^t \int_{\T} \pd_x \varphi\, 
		\pd_x\tilde{R}\,\d x\,\d t',\\
I_3&:= \int_0^t \int_{\T} \pd_x \varphi\, c(\tilde{u}) 
		\tilde{R}  \,\d x\,\d t', \qquad 
I_4 := \int_0^t \int_{\T}    \varphi \tilde{c}(\tilde{u}) 
			\big( \tilde{R} - \tilde{S} \big)^2\,\d x \,\d t',\\
I_5 & :=  \frac12 \int_0^t \int_{\T} \sigma\,\pd_x\bk{\sigma\, 
			   \varphi}\,\pd_x \bk{\tilde{R} + \tilde{S}}\,\d x\,\d t'.
\end{align*}

Moreover, we define:
\begin{equation}\label{eq:Z_n_nu_defin}
\begin{aligned}
\tilde{Z}_{n}(t) &:=- \int_{\T}  \pd_x \bk{{\bf P}_{N_n}
	\varphi  \,\sigma}\bk{\tilde{R}_{n} + \tilde{S}_{n}}\,\d x,\\
\tilde{Z} (t) &:=-  \int_{\T} 
	 \pd_x \bk{ \varphi  \,\sigma}\bk{\tilde{R} + \tilde{S}}\,\d x.
\end{aligned}
\end{equation}

The integrals $\tilde{Z}_n$ are the integrands 
in the stochastic integral in \eqref{thm:pw_convergence_N}. 
We shall show that $\tilde{Z}_n \to \tilde{Z}$ 
a.s.~strongly in $L^2([0,T])$, so that the corresponding 
stochastic integral converges in the same 
way by \cite[Lemma 2.1]{Debussche:2011aa}.
On the other hand, we also have 
$\tilde{M}_n \to \tilde{M}$ in $L^2([0,T])$.
This lets us conclude that the limiting equation holds 
for each $t \in [0,T]$, $\tilde{\mathbb{P}}$-a.s.

\begin{lem}\label{thm:pw_convergence_N}
Let $\tilde{M}_n$ be defined as in 
\eqref{eq:N_martingale2} and $\tilde{M}$, 
$\tilde{Z}_n$, and $\tilde{Z}$ be 
defined as in \eqref{eq:N_martingale3} 
and \eqref{eq:Z_n_nu_defin}.
The following $\tilde{\mathbb{P}}$-a.s.~convergences hold:
\begin{align*}
&\tilde{M}_n \xrightarrow{n \uparrow \infty} \tilde{M} 
\quad \text{and} \quad
\tilde{Z}_n \xrightarrow{n \uparrow \infty} \tilde{Z}, \quad
\text{both in $L^{2}([0,T])$}.
\end{align*}
\end{lem}

\begin{proof}
In this proof, all convergences happen 
$\tilde{\mathbb{P}}$-a.s., and we generally omit 
this epithet as understood.

From the Proposition \ref{thm:skorokhod_N}, 
$\tilde{R}_n \to \tilde{R}$ in $C([0,T];L^2(\T)-w)$ 
and $\tilde{R}^0_n \to \tilde{R}^0$ in 
$L^2(\T)$, hence 
$$
\int_{\T} \tilde{R}_{n}(t) \varphi\,\d x 
		- \int_{\T} \tilde{R}^0_n\varphi\,\d x
\to I_1 \quad \text{in $C([0,T])$}.
$$

We next argue that for $\psi_n 
	\to \psi$ in $C([0,T] \times \T)$, a.s., 

\begin{equation}\label{eq:general_n_R_n-R}
\begin{aligned}
\int_\T \psi_n \tilde{R}_n \,\d x
&\to \int_\T \psi \tilde{R} \,\d x, \\
\int_0^t\int_\T \psi_n \tilde{R}_n \,\d x  \,\d t'
&\to\int_0^t \int_\T \psi \tilde{R} \,\d x \,\d t', 
\quad \text{both in $L^2([0,T])$}.
\end{aligned}
\end{equation}
We have 
\begin{align*}
\int_0^T\abs{\int_\T \psi_n \tilde{R}_n - \psi \tilde{R}\,\d x }^2\,\d t
& \le \norm{ \bk{\psi_n - \psi}\tilde{R}_n }_{L^2_tL^1_x}^2
 +\norm{\psi \bk{\tilde{R}_n - \tilde{R}}}_{L^2_tL^1_x}^2\\
 &
\le \norm{\psi_n - \psi}_{C_{t,x}}^2 \norm{\tilde{R}_n}_{L^2_{t,x}}^2
+ \norm{ \psi}_{C_{t,x}}^2 \norm{\tilde{R}_n - \tilde{R}}_{L^2_{t,x}}^2.
\end{align*}
The first term on the right tends to nought 
as $\tilde{R}_n \to \tilde{R}$ in $L^2([0,T]\times \T)$, 
and $\norm{\tilde{R}_n(\tilde{\omega})}_{L^2_{t,x}}$ 
is hence bounded in that space uniformly in $n$. 
The second term tends to nought by 
the same convergence.
Inserting an extra temporal integral in the 
calculations above does not change the 
argument substantially, and allows us to 
deduce the second statement of \eqref{eq:general_n_R_n-R}.
The convergences \eqref{eq:general_n_R_n-R} 
also hold for $(\tilde{S}_n, \tilde{S})$ in place 
of $(\tilde{R}_n,\tilde{R})$, as these pairs share 
the same bounds and convergences.

Using the deterministic convergence 
${\bf P}_n \varphi \to \varphi$ in $H^3(\T)$, 
which implies $\pd_x {\bf P}_n \varphi 	
\to \pd_x \varphi$ in $L^\infty(\T)$, 
the first convergence of \eqref{eq:general_n_R_n-R} gives us:
\begin{align}\label{eq:Z_convergence}
\tilde{Z}_n \to \tilde{Z} 
\quad \text{in $L^2([0,T])$}.
\end{align}

Similarly, by the second statement of 
\eqref{eq:general_n_R_n-R}, we have 
\begin{align*}
&-\nu \int_0^t \int_\T \pd_{xx}^2 {\bf P}_{N_n}\varphi \,\tilde{R}_n\,\d x \,\d t' 
\to -\nu \int_0^t \int_\T \pd_{xx}^2 \varphi \,\tilde{R}\,\d x \,\d t' 
= I_2,\\
&- \frac12 \int_0^t \int_{\T} \pd_x\bk{\sigma\,\pd_x\bk{\sigma\, 
			   {\bf P}_{N_n} \varphi}}\, \bk{\tilde{R}_n + \tilde{S}_n}\,\d x\,\d t'\\
&\qquad\qquad\quad\,\,\,
 \to -\frac12 \int_0^t \int_{\T} \pd_x\bk{\sigma\,\pd_x\bk{\sigma\, 
			 \varphi}}\, \bk{\tilde{R} + \tilde{S}}\,\d x\,\d t' = I_5,			   
\end{align*}
both in $L^2([0,T])$.
The final equality in each convergence 
holds as $\tilde{R}, \tilde{S}$ take values in 
$L^2([0,T];H^1(\T))$ (Lemma \ref{thm:energy_N_SJ}), 
and it is possible to integrate-by-parts again 
after passing to the limit.

Along with the convergence 
\eqref{eq:c(u_n)_converge} of 
$c_{N_n}(\tilde{u}_n) \to c(\tilde{u})$ 
in $C([0,T]\times \T)$, \eqref{eq:general_n_R_n-R} now implies  
\begin{align*}
\int_0^t \int_\T {\bf P}_{N_n} \pd_x \varphi
	\,c_{N_n}(\tilde{u}_n) \tilde{R}_n\,\d x \,\d t' 
	\to I_3, \quad \text{in $L^2([0,T])$}.
\end{align*}

The argument for $I_4$ is similar.
The strong convergences 
$(\tilde{R}_n, \tilde{S}_n) \to (\tilde{R}, \tilde{S})$ 
in $L^2([0,T] \times \T)$ implies that  
$\bk{\tilde{R}_n - \tilde{S}_n}^2 \to \bk{\tilde{R} - \tilde{S}}^2$
in $L^1([0,T]\times \T)$. Using now the 
convergence $\tilde{c}_{N_n}(\tilde{u}_n) 
	\to \tilde{c}(\tilde{u})$ 
in $L^\infty([0,T] \times \T)$, we find 
\begin{equation*}
\begin{aligned}
&\norm{\int_0^\cdot \int_\T {\bf P}_{N_n}\varphi 
	\tilde{c}_{N_n}(\tilde{u}_N) \bk{\tilde{R}_n - \tilde{S}_n}^2 \,\d x\,\d t' 
	- I_4}_{L^2_t}\\
& \le  T \norm{ {\bf P}_{N_n}\varphi 
	\tilde{c}_{N_n}(\tilde{u}_n) }_{L^\infty_{t,x}}
	\norm{\bk{\tilde{R}_n - \tilde{S}_n}^2 
	- \bk{\tilde{R} - \tilde{S}}^2}_{L^1_{t,x}}\\
&\quad\,\, + 
	T \norm{ {\bf P}_{N_n}\varphi 
	\tilde{c}_{N_n}(\tilde{u}_n)  - \varphi 
	\tilde{c}(\tilde{u}) }_{L^\infty_{t,x}}
	\norm{\tilde{R} - \tilde{S}}_{L^2_{t,x}}^2
	\xrightarrow{n \uparrow \infty}0.
\end{aligned}
\end{equation*}

Bringing together the a.s.~convergence 
in $L^2([0,T])$ for $I_1$ to  $I_5$, 
and the convergence \eqref{eq:Z_convergence} 
for $\tilde{Z}_n$, we have proven the Lemma.

\end{proof}

Using \cite[Lemma 2.1]{Debussche:2011aa}, 
we conclude from the convergence 
$\tilde{Z}_n \to \tilde{Z}$ in Lemma 
\ref{thm:pw_convergence_N} that 
$$
\int_0^t \tilde{Z}_n \,\d W_n \to \int_0^t \tilde{Z}\,\d W 
\quad \text{$\tilde{\mathbb{P}}$-a.s.~in $L^2([0,T])$.}
$$

Using the continuity of the temporal 
integrals and the inclusion $\tilde{R} \in C([0,T];L^2(\T)-w)$, 
we conclude that for each $t$, 
$$
\tilde{M}(t) = \int_0^t \tilde{Z}\,\d W, 
	\quad \text{$\tilde{\mathbb{P}}$-a.s.}
$$
With exactly the same argument for the equation 
for $\tilde{S}$, we have shown:
\begin{thm}
With the notation of Proposition \ref{thm:skorokhod_N}, 
there exists a martingale solution 
$(\tilde{R}, \tilde{S}, \tilde{E}, \tilde{W})$ 
to the variational wave equation \eqref{eq:vvw}.
\end{thm}

\section{Existence and uniqueness of pathwise solutions}
\label{sec:pathwise}

In this section we improve our martingale existence result to 
{\em pathwise} existence via a Gy\"ongy--Krylov 
argument \cite[Section 3]{GK1996}. 
This is an SPDE version of the 
Yamada--Watanabe principle, which states 
that martingale existence and pathwise uniqueness 
implies probabilistically strong existence of solutions. 
We therefore start by showing pathwise uniqueness of 
solutions to \eqref{eq:vvw}. 
In doing so, since solutions are assumed only to lie
in $L^2_{\tilde{\omega}}L^\infty_tL^2_x$, we 
shall need to mollify \eqref{eq:vvw} with  
spatial mollifiers in order to apply It\^o's formula 
to time-continuous processes, with the mollified 
equations interpreted pointwise in $x$. 
This procedure will in turn produce standard 
and double commutators that need to be controlled (see, e.g., 
\cite[Proposition 3.4]{Punshon-Smith:2018aa}).

\subsection{Commutator estimates for $c(\tilde{u})$}
In this subsection, we establish variational wave 
equation specific convergence  
results for the composition $c(\tilde{u})$ for 
$\tilde{u}$ defined in Lemma \ref{thm:tilde_u}. 
These include commutator estimates for 
mollifications that will be used to establish 
pathwise uniqueness in Section \ref{sec:pathwise_subsec}. 
The non-zero viscosity $\nu$ allows the energy 
inequality \eqref{eq:energy_SJ} to give us $L^2_x$ control on 
$\pd_x R$ and $\pd_x S$, which come up in the nonlinear transport 
term. We shall thereby be able to deploy our 
commutator estimates below, crucially dependent on 
$H^1_x$ regularity, to send mollification to zero.

\begin{lem}\label{thm:cR1-cR2}
Recall the auxiliary function $F$ of \eqref{eq:F_defin} 
and the construction \eqref{eq:u_construct} of $\mathfrak{u}$. 
For $i = 1,2$, and $(R_i,S_i) \in L^2(\T)$, set 
\begin{align*}
u_i &= \mathfrak{u}(R_i,S_i).
\end{align*}

The following bounds hold:
\begin{itemize}
\item[(i)]
\begin{align}\label{eq:u1-u2}
\norm{u_1 - u_2}_{L^\infty_x} 
&\lesssim  \abs{F^{-1}}_{\rm Lip}
	\bk{\norm{R_1 - R_2}_{L^1_x} 
		+ \norm{S_1 - S_2}_{L^1_x}},
\end{align}
 \item[(ii)]
 \begin{equation}\label{eq:cR1-cR2}
\begin{aligned}
&\norm{c(u_1) R_1 - c(u_2) R_2}_{L^2_x}
	+ \norm{c(u_1) S_1 - c(u_2) S_2}_{L^2_x} \\
&\qquad\lesssim_\kappa \bk{1 + \norm{R_1}_{L^2_x} \wedge \norm{R_2}_{L^2_x}
	+ \norm{S_1}_{L^2_x} \wedge \norm{S_2}_{L^2_x}}\\
&\qquad\qquad \qquad\qquad\times	\bk{\norm{R_1 - R_2}_{L^2_x} 
	+ \norm{S_1 - S_2}_{L^2_x}}.
\end{aligned}
\end{equation}
\item[(iii)] If additionally, $(R_i,S_i) \in H^1(\T)$ for $i = 1,2$, 
 \begin{equation}\label{eq:cdR1-cdR2}
\begin{aligned}
&\norm{c(u_1) \pd_xR_1 - c(u_2) \pd_x R_2}_{L^2_x}
	+ \norm{c(u_1)\pd_x S_1 - c(u_2) \pd_x S_2}_{L^2_x} \\
&\qquad\lesssim_\kappa \bk{1 + \norm{\pd_x R_1}_{L^2_x} \wedge \norm{\pd_x R_2}_{L^2_x}
	+ \norm{\pd_x S_1}_{L^2_x} \wedge \norm{\pd_x S_2}_{L^2_x}}\\
&\qquad\qquad \qquad\qquad\times	\bk{\norm{R_1 - R_2}_{H^1_x} 
	+ \norm{S_1 - S_2}_{H^1_x}}.
\end{aligned}
\end{equation}
\end{itemize}
\end{lem}

\begin{rem}\label{rem:uRS_converge}
Let $J_\delta = \delta^{-1} J(x/\delta)$, $\delta > 0$, be a 
standard mollifier on $\T$. For $f \in L^1(\T)$, 
let $f_\delta := f *J_\delta$.  Let $R,S \in L^\infty([0,T];L^2(\T))$. 
Define $u^\delta := \mathfrak{u}(R_\delta, S_\delta)$ 
using \eqref{eq:u_construct}. 
Then we have, as following \eqref{eq:u_approx_defin},  
\begin{align}\label{eq:u_delta_properties}
2 c(u^\delta) \pd_x u^\delta 
= {R_\delta - S_\delta}, \qquad
\pd_x c(u^\delta) 
= 2\tilde{c}(u^\delta) \bk{R_\delta - S_\delta}.
\end{align}
Moreover using \eqref{eq:u1-u2}, 
$u^\delta \to \mathfrak{u}(R,S)$ pointwise in $(t,x)$, 
and also in $L^p([0,T];L^\infty(\T))$ 
for any $p < \infty$, by Vitali's convergence theorem.
\end{rem}

\begin{proof}
Using \eqref{eq:c_x_tilde_c_SJ} and 
and the construction \eqref{eq:u_construct}, 
\begin{align*}
\norm{u_1 - u_2}_{L^\infty_{x}} 
\lesssim \abs{F^{-1}}_{\rm Lip}  \bk{\norm{\int_0^\cdot 
	\frac{R_1 - R_2}2\,\d y}_{L^\infty_{x}} 
+ \norm{\int_0^\cdot 
	\frac{S_1 - S_2}2\,\d y}_{L^\infty_{x}}}.
\end{align*}
Since 
$$
\norm{\int_0^\cdot 
	\frac{R_1 - R_2}2\,\d y}_{L^\infty_{x}} 
\le \frac12 \int_\T \abs{R_1 - R_2}\,\d x,
$$
we have \eqref{eq:u1-u2}.

Writing
\begin{align*}
c(u_1) R_1 - c(u_2) R_2 
	& = \bk{c(u_1) - c(u_2)}R_1 + c(u_2)\bk{R_1 - R_2}, 
\end{align*}
we see by the boundedness of $c$ 
the second term on the right is controlled in 
$L^2(\T)$ by $\lesssim_\kappa\norm{R_1 - R_2}_{L^2(\T)}$. 
Since $c$ is Lipschitz, 
\begin{equation*}
\begin{aligned}
\norm{\bk{c(u_1) - c(u_2)}R_1}_{L^2_x}
& \le \norm{c(u_1) - c(u_2)}_{L^\infty_x} 
	\norm{R_1}_{L^2_x}
\lesssim_{\abs{c}_{\rm Lip}}
	\norm{u_1 - u_2}_{L^\infty_x}
	\norm{R_1}_{L^2_x}.
\end{aligned}
\end{equation*}
Inserting \eqref{eq:u1-u2} to control the 
difference $\norm{u_1 - u_2}_{L^\infty(\T)}$
gives us \eqref{eq:cR1-cR2}.

The same argument with $(\pd_x R_i, \pd_x S_i)$, 
in place of $(R_i, S_i)$, $i = 1,2$, implies the third 
bound \eqref{eq:cdR1-cdR2}.
\end{proof}

We prove one further commutator estimate 
involving the nonlinearity $c$ that we shall 
use directly in the proof of Theorem \ref{thm:pathwiseunique}.
\begin{lem}\label{thm:c_commutator}
Let $J_\delta$ be a standard mollifier on $\T$. 
Let $R, S \in L^\infty([0,T];L^2(\T)) 
	\cap L^2([0,T]; H^1(\T))$, and set 
$R_\delta := R*J_\delta$, $S_\delta := S *J_\delta$. 
Let $u := \mathfrak{u}(R,S)$ be defined as in \eqref{eq:u_construct} 
and set $u^\delta := \mathfrak{u}(R_\delta, S_\delta)$. 
The following commutator estimates hold:
\begin{align}
\int_0^T\int_{\T} \abs{c(u^\delta) \pd_x R_\delta 
	- \bk{c(u) \pd_x R} *J_\delta}^2\,\d x \,\d t 
&= o_{\delta \downarrow 0}(1), \label{eq:c_commute1}\\
 \int_0^T\int_{\T} \abs{\tilde{c}(u^\delta)
	\bk{R_{\delta} - S_{\delta}} R_\delta - 
	\bk{\tilde{c}(u) \bk{R - S}R} *J_\delta}^2\,\d x \,\d t
& = o_{\delta \downarrow 0}(1) \label{eq:c_commute2}.
\end{align}
\end{lem}

\begin{rem}\label{rem:c_commute34}
Similarly, we have 
\begin{align*}
\int_0^T\int_{\T} \abs{c(u^\delta) \pd_x S_\delta 
	- \bk{c(u) \pd_x S} *J_\delta}^2\,\d x \,\d t'
&= o_{\delta \downarrow 0}(1), \\
\int_0^T \int_{\T} \abs{\tilde{c}(u^\delta)
	\bk{R_{\delta} - S_{\delta}} S_\delta - 
	\bk{\tilde{c}(u) \bk{R - S}S} *J_\delta}^2\,\d x \,\d t'
& = o_{\delta \downarrow 0}(1). 
\end{align*}
\end{rem}

\begin{proof}

{\em 1. Proof of \eqref{eq:c_commute1}.}
\smallskip

We split the integrand into
\begin{align*}
 c&(u^\delta) \pd_x R_\delta - \bk{c(u) \pd_x R} *J_\delta\\
& =  \underbrace{c(u^\delta) \pd_x R_\delta - c(u^\delta) \pd_x R}_{=:I_1}
 +  \underbrace{c(u^\delta) \pd_x R  - c(u) \pd_x R}_{=:I_2} 
 + \underbrace{c(u) \,\pd_x R- \bk{c(u) \pd_x R} *J_\delta}_{=:I_3}.
\end{align*}
Using the standard properties of 
mollifiers and the inclusion
$\pd_x R \in L^2([0,T]\times \T)$, $\norm{I_3}_{L^2_{t,x}} 
	= o_{\delta\downarrow 0}(1)$, and 
\begin{align*}
\norm{I_1}_{L^2_{t,x}}
	\le \norm{c}_{L^\infty_{t,x}} 
	\norm{\pd_x R_\delta - \pd_x R}_{L^2_{t,x}}
	\xrightarrow{\delta \downarrow 0} 0.
\end{align*}
Using Remark \ref{rem:uRS_converge} and 
the Lipschitz bound $\abs{c'} < \kappa$, 
$u^\delta \to u = \mathfrak{u}(R,S)$ implies  
$c(u^\delta) \to c(u)$ in $L^\infty([0,T] \times \T)$. 
Therefore, $\norm{I_2}_{L^2_{t,x}} 	
	= o_{\delta \downarrow 0}(1)$. 

\medskip
{\em 2. Proof of \eqref{eq:c_commute2}.}
\smallskip

We first observe that 
\begin{align*}
\norm{\tilde{c}(u) \bk{R - S} R }_{L^2_{t,x}} 
\le \kappa \norm{R - S}_{L^\infty_t L^2_x}\norm{R}_{L^2_t H^1_x}, 
\end{align*}
so $\tilde{c}(u) \bk{R - S} R \in L^2([0,T] \times T)$, 
and therefore we only need to check that 
$\tilde{c}(u^\delta) \bk{R_\delta - S_\delta} R_\delta 
	- \tilde{c}(u) \bk{R - S} R$ 
tends to zero in $L^2([0,T]\times \T)$ as $\delta \downarrow 0$.

Splitting the difference as
\begin{align*}
&\norm{\tilde{c}(u^\delta) \bk{R_\delta - S_\delta} R_\delta 
- \tilde{c}(u) \bk{R - S} R}_{L^2_{t,x}}\\
& \le \underbrace{\norm{\tilde{c}(u^\delta) 
	\big(\bk{R_\delta - S_\delta} R_\delta
	 - \bk{R - S} R\big)}_{L^2_{t,x}}}_{=:I_1}
 +\underbrace{\norm{ \big(\tilde{c}(u^\delta)  
- \tilde{c}(u)\big) \bk{R - S} R}_{L^2_{t,x}}}_{=:I_2} ,
\end{align*}
we can estimate the terms separately.

For $I_1$, we use the uniform bound on 
$\tilde{c}$ to get 
\begin{align*}
I_1 &\le \kappa^2 \norm{\big(\bk{R_\delta - S_\delta}
		 - \bk{R - S} \big)R}_{L^2_{t,x}}
	+ \kappa^2 \norm{\bk{R_\delta - S_\delta}
		 \bk{R_\delta - R}}_{L^2_{t,x}}\\
	& \lesssim_\kappa  \norm{\bk{R_\delta - S_\delta}
		 - \bk{R - S}}_{L^2_tL^\infty_x}
		\norm{R}_{L^\infty_tL^2_x}
		+  \norm{R_\delta - S_\delta}_{L^\infty_tL^2_x}
	\norm{R_\delta - R}_{L^2_tL^\infty_x}\\
	& \lesssim_\kappa  \norm{\bk{R_\delta - S_\delta}
		 - \bk{R - S}}_{L^2_tH^1_x}
		\norm{R}_{L^\infty_tL^2_x}
		 +  \norm{R_\delta - S_\delta}_{L^\infty_tL^2_x}
	\norm{R_\delta - R}_{L^2_tH^1_x}.
\end{align*}
Each of the terms in $L^2_tH^1_x$ norms 
tend to zero by the standard properties of 
mollification; all other quantities are bounded.

The $(t,x)$ pointwise convergence 
$u^\delta \to u$ (Remark \ref{rem:uRS_converge}) 
and the continuity of $\tilde{c}$ imply that 
$f_\delta := \bk{\tilde{c}(u^\delta) - \tilde{c}(u)}^2\to 0$ 
pointwise on $[0,T] \times \T$. Let $\overline{C}$ 
be the $L^\infty_{t,x}$ weak* limit of a given  
subsequence of the uniformly bounded sequence 
$\{f_\delta\}$. Then
$$
\forall \varphi \in L^1_{t,x},\qquad 
\int_0^T \int_\T \varphi f_\delta \,\d x\,\d t  
\to \int_0^T \int_\T \varphi \overline{C} \,\d x\,\d t. 
$$
On the other hand, by the majorisation 
$\abs{\varphi f_\delta} \le 4\kappa^4 \abs{\varphi} \in L^1_{t,x}$, 
$\int_0^T \int_\T \varphi f_\delta \,\d x \,\d t \to 0$ 
by the dominated convergence theorem. Therefore 
$\overline{C} \equiv 0$ a.e., along any sub-subsequence 
and hence along the entire sequence.

Now $\bk{R - S}^2R^2 \in L^1_{t,x}$ as 
$\bk{R - S}^2 \in L^\infty_tL^1_x$ 
and $R^2 \in L^1_{t} L^\infty_x$ (since 
$R \in L^2_tH^1_x \hookrightarrow L^2_tL^\infty_x$). 
Therefore,
\begin{align*}
I_2^2 = \int_0^T \int_\T 
	\bk{\tilde{c}(u^\delta) - \tilde{c}(u)}^2
	\bk{R - S}^2R^2\,\d x \,\d t \to 0.
\end{align*}
\end{proof}

\subsection{Pathwise uniqueness}\label{sec:pathwise_subsec}
In order to prove our uniqueness result, we 
shall use the following stochastic Gronwall 
lemma, which marginally generalises  \cite[Lemma 3.8]
	{Xie:2020vh} and \cite[Theorem 4]{Scheutzow:2013wy}
to the case of stopping times.
\begin{lem}[Stochastic Gronwall inequality 
	{\cite[Lemma A.2]{HKP2023}}]	
	\label{thm:sto_gronwall_st}
For a given filtered probability space, 
let $\xi(t)$ and $\eta(t)$ be two non-negative 
adapted processes, $A(t)$ be a continuous, adapted, 
non-decreasing process with $A(0)=0$,
and $M$ a local martingale with $M(0) = 0$. 
Let $\tau$ be a
stopping time on the same filtration as $M$ is
a martingale. Suppose $\xi$ is c\`adl\`ag in time 
and satisfies the following stochastic 
differential inequality on $[0,T\wedge \tau]$:
$$
\d \xi \le \eta \,\d t + \xi \,\d A + \,\d M.
$$
For $0 < \nu <  r < 1$, we have
\begin{align*}
	\bigg(\Ex &\sup_{s \in [0, T \wedge \tau]}
	\abs{\xi(s)}^\nu\bigg)^{1/\nu} \\
	&\le \bk{\frac{r}{r-\nu}}^{1/\nu}
	\bigg(\Ex \exp\left(\frac{rA(T 
		\wedge \tau)}{1-r} \right)\bigg)^{(1-r)/r}
	\Ex \bk{\xi(0) + \int_0^{T\wedge \tau} \eta(s) \,\d s}.
\end{align*}
\end{lem}

The main result of this subsection is:
\begin{thm}[Pathwise uniqueness]\label{thm:pathwiseunique}
Let $(R_1, S_1)$ and $(R_2,S_2)$ be pathwise   
solutions to \eqref{eq:vvw}, both with initial conditions 
$(R^0, S^0) \in \bk{L^{2p_0}(\Omega;L^2(\T))}^2$. Then 
\begin{align*}
\Ex \sup_{t \in [0,T]}&\bk{\norm{R_1 - R_2}_{L^2(\T)}^2  
	+ \norm{S_1 - S_2}_{L^2(\T)}^2 }^{1/2} = 0.
\end{align*}
\end{thm}

\begin{proof}
As in Remark \ref{rem:uRS_converge}, let $J_\delta$ 
be a standard mollifier on $\T$ indexed by $\delta > 0$. 
Set $f_\delta := f *J_\delta$ for any $f \in L^1(\T)$. 
Let $\mathfrak{u}(R,S,z)$ be as in 
\eqref{eq:u_construct}, and for $i =1,2$, define
$$
u_i := \mathfrak{u}(R_i,S_i),\qquad 
u_i^\delta := \mathfrak{u}(R_{i,\delta},S_{i,\delta}).
$$
Let us also employ the shorthand 
$$
V_i := R_i + S_i \qquad \text{for $i = 1,2$}.
$$
Since $R_i$ and $S_i$ satisfy the same bounds, 
so must $V_i$. 

By testing each equation for $R_1$, $S_1$, 
$R_2$, and $S_2$ (cf.~\eqref{eq:vvw}) against 
$J_\delta(x - \cdot)$, we get:
\begin{align}\label{eq:R_difference_viscous}
\d \bk{R_{1,\delta} - R_{2,\delta}} 
	&= \sum_{j = 1}^4 I^R_j \,\d t +I_5^R \,\d W  
	+ \sum_{j = 1}^3E_{j,\delta}^R \,\d t 
		+ E_{4,\delta}^R\,\d W, 
\end{align}
where 
\begin{align*}
I_1^R &:= \pd_x \bk{c(u_1^\delta) R_{1,\delta} 
	- c(u_2^\delta) R_{2,\delta}}, \qquad\qquad\quad\,
I_2^R := \nu \pd_{xx}^2 \bk{ R_{1,\delta} - R_{2,\delta}},\\
I_3^R &:= - \tilde{c}(u_1^\delta)\bk{R_{1,\delta} - S_{1,\delta}}^2
	+  \tilde{c}(u_2^\delta)\bk{R_{2,\delta} - S_{2,\delta}}^2,\\
I_4^R &:=  \sigma \,\pd_x \bk{\sigma \,\pd_x 
	 \bk{V_{1,\delta} - V_{2,\delta}}}, \qquad\qquad\qquad\,\,\,
I_5^R :=  \sigma\,\pd_x \bk{V_{1,\delta} - V_{2,\delta}}, \\
E_{1,\delta}^R &:=  \bk{c(u_1) R_{1} 
	- c(u_2) R_{2}}*\pd_xJ_\delta - I_1^R, \qquad 
E_{2,\delta}^R := 0 , \\
E_{3,\delta}^R &:= - \bk{ \tilde{c}(u_1)\bk{R_{1} - S_{1}}^2
	-  \tilde{c}(u_2)\bk{R_{2} - S_{2}}^2}*J_\delta - I_3^R,\\
E_{4,\delta}^R &:=   \frac12 \bk{\sigma \,\pd_x 
	\bk{\sigma\,\pd_x \bk{V_{1} - V_{2}}}}*J_\delta - I_4^R,\quad\,\,\,\,
E_{5,\delta}^R :=  \bk{\sigma\,
	\pd_x \bk{V_{1} - V_{2}}}*J_\delta - I_5^R.
\end{align*}

Similarly, for the $S$ equation, we have
\begin{align}\label{eq:S_difference_viscous}
\d \bk{S_{1,\delta} - S_{2,\delta}} 
	&= \sum_{j = 1}^4 I^S_j \,\d t +I_5^S \,\d W  
	+ \sum_{j = 1}^3E_{j,\delta}^S \,\d t 
		+ E_{4,\delta}^S\,\d W, 
\end{align}
with 
\begin{align*}
I_1^S &:= - \pd_x \bk{c(u_1^\delta) S_{1,\delta} 
	- c(u_2^\delta) S_{2,\delta}}, \qquad
I_2^S := \nu \pd_{xx}^2 \bk{ S_{1,\delta} - S_{2,\delta}},\\
I_3^S &:= I_3^R, \qquad 
I_4^S := I_4^R, \qquad
I_5^S :=  I_5^R, \\
E_{1,\delta}^S &:= - \bk{c(u_1) S_{1} 
	- c(u_2) S_{2}}*\pd_xJ_\delta - I_1^S, \\ 
E_{2,\delta}^S &:= 0 , \qquad
E_{3,\delta}^S := E_{3,\delta}^R,\qquad
E_{4,\delta}^S :=   E_{4,\delta}^R,\qquad 
E_{5,\delta}^S :=  E_{5,\delta}^R.
\end{align*}

It will be important in our analysis to treat 
the error terms $E_{4,\delta}^R$,   $E_{4,\delta}^S$, 
$E_{5,\delta}^R$, and $E_{5,\delta}^S$ together. 

Applying It\^o's formula for the nonlinearity 
$v \mapsto \frac12 v^2$ to \eqref{eq:R_difference_viscous}, 
and then integrating in $x$, we get:
\begin{equation}\label{eq:itoformula_R-R_2}
\begin{aligned}
&\frac12\d \norm{R_{1,\delta} - R_{2,\delta}}_{L^2_x}^2\\
&	= \sum_{j = 1}^4 \int_{\T} 
	\bk{R_{1,\delta} - R_{2,\delta}} I^R_j \,\d x \,\d t 
	 + \int_{\T}\bk{R_{1,\delta} - R_{2,\delta}} I_5^R \,\d x \,\d W  \\
&\quad\,\,+ \sum_{j = 1}^4 \int_{\T}
	 \bk{R_{1,\delta} - R_{2,\delta}}E_{j,\delta}^R\,\d x  \,\d t 
		+ \int_{\T}\bk{R_{1,\delta} - R_{2,\delta}} 
			E_{5,\delta}^R\,\d x \,\d W\\
& \quad\,\, + \frac12 \int_{\T}  
		\abs{ I_5^R + E_{5,\delta}^R}^2\,\d x \,\d t, 
\end{aligned}
\end{equation}

The remainder of this fairly lengthy proof 
is composed of the following steps:
\begin{enumerate}
\item Bound terms in \eqref{eq:itoformula_R-R_2}
	 involving $I^R_1$, $I^R_2$, $I^R_3$. 
	These are the main terms from the ``deterministic 
	part'' of the $R$-equation.
	Entirely analogous bounds for corresponding  
	terms in the $S$-equation also hold. 
	These bounds must be added together 
	to close estimates.	
\item Bound error terms in \eqref{eq:itoformula_R-R_2}
	involving $E^R_{1,\delta}$, $E^R_{3,\delta}$.

\item Show how terms in \eqref{eq:itoformula_R-R_2} 
	and corresponding terms 
	of the $S$-equation,  
	involving $I^R_4$ and $I^S_4$, arising from the 
	Stratonovich-to-It\^o noise conversion 
	combine to produce terms 
	that can be bounded. 

\item Bound error terms involving 
	$E^R_{4,\delta}$, $E^S_{4,\delta}$, 
	and the It\^o correction terms involving 
	$\abs{I^R_5 + E^R_{5,\delta}}^2$ and 
	$\abs{I^S_5 + E^S_{5,\delta}}^2$. 
	They need 
	to be combined with one another 
	and also with the estimates for 
	terms involving  $I^R_4$ and $I^S_4$ 
	in order properly to vanish.	 
	
\item Introduce a stopping time indexed 
	by a parameter $L > 0$, up to 
	which $\norm{R_{1,\delta} 
		- R_{2,\delta}}_{L^\infty_tL^2_x} + 
	\norm{S_{1,\delta} - S_{2,\delta}}_{L^\infty_tL^2_x}$ 
	can be controlled
	via Lemma \ref{thm:sto_gronwall_st}. 
	And finally, give estimates on this stopping time 
	and take the limits $\delta \downarrow 0$, 
	$L \uparrow \infty$ to conclude.
\end{enumerate}

\smallskip
{\em Step 1: Bounding terms involving $I^R_1$, $I^R_2$, and $I^R_3$. }

The dissipation is non-positive:
\begin{align}\label{eq:pathwise_I2}
\int_0^t \int_{\T}\bk{R_{1,\delta} - R_{2,\delta}}I_2^R\,\d x \,\d t'
= - \nu \int_0^t \int_{\T} \abs{\pd_x 
	\bk{R_{1,\delta} - R_{2,\delta}}}^2\,\d x\,\d t' \le 0.
\end{align}
We shall use this term to absorb the gradient  
term in \eqref{eq:pathwise_I1I2I3} below,  
emphasising that the current pathwise 
uniqueness result is strictly a result of the 
presence of non-zero viscosity.

By Lemma \ref{thm:cR1-cR2}, integrating by parts and using Young's inequality, 
\begin{equation}\label{eq:J1J2}
\begin{aligned}
&\int_0^t\abs{\int_{\T}\bk{R_{1,\delta} - R_{2,\delta}}I_1^R\,\d x}\d t' \\
&\lesssim \int_0^t \bk{1 + \norm{R_{1,\delta}}_{L^2_x} 
	+ \norm{S_{1,\delta}}_{L^2_x}}\\
&\qquad\quad \times \frac{2}{\sqrt{\nu}}
	\bk{\norm{R_{1,\delta} - R_{2,\delta}}_{L^2_x} 
	+ \norm{S_{1,\delta} + S_{2,\delta}}_{L^2_x}}\\
&\qquad \quad	\times \frac{\sqrt{\nu}}{2}
	\norm{\pd_xR_{1,\delta} - \pd_x R_{2,\delta}}_{L^2_x}\,\d t'\\
& \le \frac2\nu  \int_0^t 	\bk{1 + \norm{R_{1,\delta}}_{L^2_x} 
	+ \norm{S_{1,\delta}}_{L^2_x}}^2\\
&\qquad\qquad \times  	
		\bk{\norm{R_{1,\delta} - R_{2,\delta}}_{L^2_x}^2
		+\norm{S_{1,\delta} - S_{2,\delta}}_{L^2_x}^2}\,\d t' \\
&\qquad	+ \frac\nu8 \int_0^t
		 \norm{\pd_xR_{1,\delta} - \pd_x R_{2,\delta}}_{L^2_x}^2\,\d t'.
\end{aligned}
\end{equation}

Using \eqref{eq:u_delta_properties}, and integrating by parts, 
for $I_3^R$ we find
\begin{align*}
&2\int_{\T} \bk{R_{1,\delta} - R_{2,\delta}} I_3^R \,\d x\\
& = -\int_{\T} \pd_x c(u_1^\delta) 
	\bk{R_{1,\delta} + S_{1,\delta}}
	\bk{R_{1,\delta} - R_{2,\delta}}\,\d x \\
&\quad\,\, + \int_{\T} \pd_x c(u_2^\delta) 
	\bk{R_{2,\delta} + S_{2,\delta}}
	\bk{R_{1,\delta} - R_{2,\delta}}\,\d x\\
& = \int_{\T} c(u_1^\delta) 
	 \pd_x\bk{R_{1,\delta} + S_{1,\delta}}
	 \bk{R_{1,\delta} - R_{2,\delta}}\,\d x \\
&\quad\,\, +  \int_{\T} c(u_1^\delta)
	\bk{R_{1,\delta} + S_{1,\delta}}  
	\pd_x\bk{R_{1,\delta} - R_{2,\delta}}\,\d x \\
&\quad\,\, - \int_{\T} c(u_2^\delta) 
	\pd_x \bk{R_{2,\delta} + S_{2,\delta}}
	\bk{R_{1,\delta} - R_{2,\delta}}\,\d x\\
&\quad\,\, - \int_{\T} c(u_2^\delta) 
	\bk{R_{2,\delta} + S_{2,\delta}} 
	\pd_x\bk{R_{1,\delta} - R_{2,\delta}}\,\d x\\
& =  \int_{\T} \bk{c(u_1^\delta)	\pd_x R_{1,\delta} 
		- c(u_2^\delta) \pd_x R_{2,\delta}} 
	\bk{R_{1,\delta} - R_{2,\delta}}\,\d x \\
&\quad\,\, +  \int_{\T} \bk{c(u_1^\delta)\pd_x S_{1,\delta} 
		- c(u_2^\delta) \pd_x S_{2,\delta}  }  
	\bk{R_{1,\delta} - R_{2,\delta}}\,\d x\\
&\quad\,\, +  \int_{\T} \bk{c(u_1^\delta)R_{1,\delta} 
		- c(u_2^\delta) R_{2,\delta}  }  \pd_x
	\bk{R_{1,\delta} - R_{2,\delta}}\,\d x \\
&\quad\,\, + \int_{\T} \bk{c(u_1^\delta)S_{1,\delta} 
		- c(u_2^\delta) S_{2,\delta}  }  
	\pd_x\bk{R_{1,\delta} - R_{2,\delta}}\,\d x.
\end{align*}
We recognise within each integrand 
above factors whose $L^2_x$ norms are 
controlled by \eqref{eq:cR1-cR2} and \eqref{eq:cdR1-cdR2}. 
Bringing these inequalities to bear, 
and again deploying Young's inequality, we find that 
for some deteministic constant $C_\nu$,
\begin{align*}
&\int_0^t \abs{\int_{\T} \bk{R_{1,\delta} - R_{2,\delta}} I_3^R \,\d x}\,\d t' \\
&\le \int_0^t \norm{R_{1,\delta} - R_{2,\delta}}_{L^2_x}
	\bk{1 +  \norm{R_1}_{H^1_x} + \norm{S_1}_{H^1_x}}\\
&\qquad\quad\times\bk{\norm{R_{1,\delta} - R_{2,\delta}}_{H^1_x} 
	+ \norm{S_{1,\delta} - S_{2,\delta}}_{H^1_x}}\,\d t'\\
&\quad\,\, + \int_0^t \norm{R_{1,\delta} - R_{2,\delta}}_{H^1_x}
	\bk{1 +  \norm{R_1}_{L^2_x} + \norm{S_1}_{L^2_x}}\\
&\qquad\quad\times\bk{\norm{R_{1,\delta} - R_{2,\delta}}_{L^2_x} 
	+ \norm{S_{1,\delta} - S_{2,\delta}}_{L^2_x}}\,\d t'\\
&\le C_\nu \int_0^t \bk{1 + \norm{R_1(t')}_{H^1_x} 
	+ \norm{S_1(t')}_{H^1_x}}^2\\
	&\qquad \qquad \qquad \times \bk{  \norm{R_1 - R_2}_{L^2_x}^2 
			+ \norm{S_1 - S_2}_{L^2_x}^2}\,\d t'\\
&\qquad +\frac\nu8 \int_0^t
		 \norm{\pd_xR_{1,\delta} - \pd_x R_{2,\delta}}_{L^2_x}^2
		 + \norm{\pd_xS_{1,\delta} - \pd_x S_{2,\delta}}_{L^2_x}^2\,\d t'.
\end{align*}

Together with \eqref{eq:pathwise_I2} and 
\eqref{eq:J1J2}, we get, for a slightly bigger  
deteministic constant $C_\nu$,
\begin{align}
&\int_0^t \int_{\T}\bk{R_{1,\delta} - R_{2,\delta}}
	\bk{I_1^R + I_2^R + I_3^R} \,\d x\,\d t' \notag\\
&\le C_\nu \int_0^t \bk{1 + \norm{R_1(t')}_{H^1_x} 
	+ \norm{S_1(t')}_{H^1_x}}^2\notag\\
	&\qquad \qquad \qquad \times \bk{  \norm{R_1 - R_2}_{L^2_x}^2 
			+ \norm{S_1 - S_2}_{L^2_x}^2}\,\d t'\label{eq:pathwise_I1I2I3}\\
&\qquad +\frac\nu4 \int_0^t
		 \norm{\pd_xR_{1,\delta} - \pd_x R_{2,\delta}}_{L^2_x}^2
		 + \norm{\pd_xS_{1,\delta} - \pd_x S_{2,\delta}}_{L^2_x}^2\,\d t'\notag\\
&\qquad - \nu \int_0^t
		 \norm{\pd_xR_{1,\delta} - \pd_x R_{2,\delta}}_{L^2_x}^2\,\d t'.\notag
\end{align}
A calculation repeating the manipulations 
above for $I_1^S$, $I_2^S$, and $I_3^S$ 
shows that \eqref{eq:pathwise_I1I2I3} 
holds when the symbols ``$R$'' and ``$S$'' are swapped. 
This gives us 
\begin{equation}\label{eq:pathwise_SRI1I2I3}
\begin{aligned}
&\int_0^t \int_{\T}\bk{R_{1,\delta} - R_{2,\delta}}
	\bk{I_1^R + I_2^R + I_3^R} \,\d x\,\d t' \\
& + \int_0^t \int_{\T}\bk{S_{1,\delta} - S_{2,\delta}}
	\bk{I_1^S + I_2^S + I_3^S} \,\d x\,\d t' \\
&\le C_\nu \int_0^t \bk{1 + \norm{R_1(t')}_{H^1_x} 
	+ \norm{S_1(t')}_{H^1_x}}^2\\
	&\qquad \qquad \qquad \times \bk{  \norm{R_1 - R_2}_{L^2_x}^2 
			+ \norm{S_1 - S_2}_{L^2_x}^2}\,\d t'.
\end{aligned}
\end{equation}

\smallskip
{\em 2. Commutator terms $E^R_{1,\delta}$ and $E^R_{3,\delta}$.}

We now turn to the commutator errors $E^R_{j,\delta}$, $j = 1,3$. 
By the energy bound \eqref{eq:energy_SJ} 
and uniform boundedness of the nonlinear 
function $c$, we have the bounds:
\begin{align*}
\Ex \norm{E^R_{1,\delta} }_{L^2_{t,x}}^{p_0} 
\lesssim 1,\qquad 
\Ex \norm{E^R_{3,\delta}}_{L^2_{t,x}}^{p_0} 
\lesssim 1.
\end{align*}
Expanding the derivative in $E^R_{1,\delta}$ and using 
\eqref{eq:u_delta_properties}, we can write  
\begin{align*}
E^R_{1,\delta} &= \bk{c(u_1)\, \pd_x R_1 - c(u_2)\, \pd_x R_2 }*J_\delta 
-\bk{ c(u_1^\delta) \,\pd_x R_{1,\delta} - c(u_2^\delta)
	 \,\pd_x R_{2,\delta}}\\
&\quad\,\, + 2\bk{\tilde{c}(u_1) \bk{R_1 - S_1} R_1 
	- \tilde{c}(u_2)\bk{R_2 - S_2} R_2}*J_\delta\\
&\quad\,\, - 2 \bk{\tilde{c}(u_1^\delta)
	 \bk{R_{1,\delta} - S_{1,\delta}} R_{1,\delta} 
	- \tilde{c}(u_2^\delta)\bk{R_{2,\delta} - S_{2,\delta}} R_{2,\delta}}.
\end{align*}
Using the a.s.~convergence 
of Lemma \ref{thm:c_commutator} and the Vitali 
convergence theorem, we find
\begin{align}\label{eq:pathwise_E1E3}
 E^R_{1,\delta}, \, E^R_{3,\delta}
 \xrightarrow{\delta \downarrow 0} 0 
\quad \text{ in $L^2(\Omega \times [0,T] \times \T)$.}
\end{align}

\smallskip
{\em 3. Estimates for terms involving 
	$I_4^R$, $I_4^S$.}

The terms with $I_4^R$ and $I_4^S$ arise from the conversion of 
the Stratonovich integral to the It\^o integral. 
It will be used to cancel a term involving 
$\abs{I_5^R}^2 + \abs{I_5^S}^2$ in \eqref{eq:doublecommutator} below. 
We now avail ourselves of the fact that $I_4^R = I_4^S$, 
and of the shorthand $V_i = R_i + S_i$ for $i = 1,2$. Then
\begin{equation}\label{eq:pathwise_I4}
\begin{aligned}
&  \int_0^t \int_{\T}
	 \bk{R_{1,\delta} - R_{2,\delta}} I_4^R \,\d x\,\d t'
+   \int_0^t \int_{\T}
	 \bk{S_{1,\delta} - S_{2,\delta}} I_4^S \,\d x\,\d t'\\
& =  \int_0^t \int_{\T}
	 \bk{V_{1,\delta} - V_{2,\delta}} I_4^R \,\d x\,\d t'\\
& = -    \int_0^t \int_{\T} \abs{\sigma \,\pd_x 
	 \bk{V_{1,\delta} - V_{2,\delta}} }^2\,\d x \,\d t'+ \frac12  \int_0^t \int_{\T}
 	\pd_{xx}^2 \sigma^2  \abs{V_{1,\delta} - V_{2,\delta}}^2 \,\d x \,\d t'\\
& \le  -\frac12  \int_0^t \int_{\T} \abs{I_5^R}^2 + \abs{I_5^S}^2\,\d x \,\d t' 
	 + \frac18\norm{\pd_{xx}^2 \sigma^2}_{L^\infty_x}
	 \int_0^t \norm{V_{1,\delta} - V_{2,\delta}}_{L^2_x}^2\,\d t'.
\end{aligned}
\end{equation}

\smallskip
{\em 4. Commutator terms and the It\^o correction: 
	$I_5^R$, $I_5^S$, $E_{4,\delta}^R$, 
	$E_{4,\delta}^S$,  $E_{5,\delta}^R$, $E_{5,\delta}^S$.}

Let us being by recalling that 
$I_5^R = I_5^S$, $E_{4,\delta}^R = E_{4,\delta}^S$,  
and $E_{5,\delta}^R = E_{5,\delta}^S$.

The error $E^R_{5,\delta} $ is a standard commutator 
term controlled by \cite[Lemma II.1]{DiPerna:1989aa}:
\begin{align}\label{eq:commutator_1_dl}
E^R_{5,\delta} \xrightarrow{\delta \downarrow 0} 0 
\quad \text{ in $L^2(\Omega \times [0,T] \times \T)$.}
\end{align}
It turns out that whilst $E^R_{4,\delta} $ does not necessarily 
vanish by itself, a specific combination of $E^R_{4,\delta}$ 
and $E^R_{5,\delta}$ does \cite[Proposition 3.4]
	{Punshon-Smith:2018aa}
(see also \cite[Lemma 7.1, Proposition 7.4]{HKP2023} and 
\cite[Lemma 1, Theorem 1]{Pan2023} for the non 
divergence-free cases, and \cite[Appendix A]{GK2021} 
for the renormalisation theory of transport noises 
on compact Riemannian manifolds): 
\begin{align*}
\Ex \int_0^t \abs{\int_{\T} 
	\bk{V_{1,\delta} - V_{2,\delta}} E^R_{4,\delta}
	 + 2 I_5^R E^R_{5,\delta}\,\d x } \,\d t' 
\xrightarrow{\delta \downarrow 0} 0.
\end{align*}

Therefore, finally, along with 
\eqref{eq:commutator_1_dl}, 
\begin{equation}\label{eq:doublecommutator}
\begin{aligned}
& \int_0^t \int_{\T}
	 \bk{R_{1,\delta} - R_{2,\delta}}E_{4,\delta}^R\,\d x  \,\d t'  
	 + \frac12 \int_{\T} 
\abs{ I_5^R + E_{5,\delta}^R}^2\,\d x \,\d t'\\
&  +\int_0^t \int_{\T}
	 \bk{S_{1,\delta} - S_{2,\delta}}E_{4,\delta}^S\,\d x  \,\d t'  
	 + \frac12 \int_{\T} 
\abs{ I_5^S + E_{5,\delta}^S}^2\,\d x \,\d t'\\
& = \frac12  \int_0^t\int_{\T }\abs{I_5^R}^2  + \abs{I_5^S}^2 \,\d x\,\d t' 
	+  \int_0^t\int_{\T }
		\abs{E^R_{5,\delta}}^2 \,\d x\,\d t'\\
&\qquad	+  \int_0^t\int_{\T }
	\bk{V_{1,\delta} - V_{2,\delta}}E_{4,\delta}^R 
		+ 2 I_5^R E^R_{5,\delta} \,\d x \,\d t' \\
& \le  
    \frac12  \int_0^t\int_{\T }\abs{I_5^R}^2  + \abs{I_5^S}^2 \,\d x\,\d t' 
	+   \int_0^t\int_{\T }
		\abs{E^R_{5,\delta}}^2 \,\d x\,\d t'\\
&\qquad	+  \int_0^t\abs{\int_{\T }
	\bk{V_{1,\delta} - V_{2,\delta}}E_{4,\delta}^R 
		+2 I_5^R E^R_{5,\delta} \,\d x} \,\d t' \\
& =  \frac12  \int_0^t\int_{\T }\abs{I_5^R}^2  + \abs{I_5^S}^2 \,\d x\,\d t' 
	+ h_\delta(t),
\end{aligned}
\end{equation}
where we used $E_{4,\delta}^R = E_{4,\delta}^S$ 
and $E_{5,\delta}^R = E_{5,\delta}^S$, and 
$h_\delta$ is an a.s. increasing function 
in time and $\Ex \abs{h_\delta(T)} = o_{\delta \downarrow 0}(1)$. 
By combining \eqref{eq:doublecommutator} 
with \eqref{eq:pathwise_I4} above, we shall 
be able to get rid of the $\abs{I^R_5}^2 + \abs{I^S_5}^2$ 
term which diverges in the $\delta \downarrow 0$ limit.

\smallskip
{\em 5. Stopping time and the stochastic Gronwall inequality}

Set 
$$
\xi_\delta(t) 
:= \norm{R_{1,\delta}(t) - R_{2,\delta}(t)}_{L^2(\T)}^2 
+  \norm{S_{1,\delta}(t) - S_{2,\delta}(t)}_{L^2(\T)}^2.
$$
By the energy bound \eqref{eq:energy_SJ} and the 
standard properties of convolutions, 
$\xi_\delta \to \xi :=  \norm{R_{1}(t) - R_{2}(t)}_{L^2(\T)}^2 
+  \norm{S_{1}(t) - S_{2}(t)}_{L^2(\T)}^2$, $(\omega, t)$-a.e.
and in $L^{p_0}_{\omega, t}$.

Gathering the estimates \eqref{eq:pathwise_SRI1I2I3} --
\eqref{eq:pathwise_I4} and  \eqref{eq:doublecommutator}, 
we finally arrive at:
\begin{equation*}
\begin{aligned}
\xi_\delta(t)
&\le C_\nu \int_0^t \bk{1 + \norm{R_1(t')}_{H^1_x} 
 	+ \norm{S_1(t')}_{H^1_x}}^2  \xi_\delta(t')\,\d t'\\
& \qquad \qquad + \frac18\norm{\pd_{xx}^2 \sigma^2}_{L^\infty_x}
	 \int_0^t \underbrace{\norm{V_{1,\delta} 
	 	- V_{2,\delta}}_{L^2_x}^2}_{\le \xi_\delta(t')}\,\d t'
	+ h_\delta(t) + M(t), 
\end{aligned}
\end{equation*}
where $M(t)$ is a continuous, 
square-integrable martingale 
(by It\^o isometry and \eqref{eq:energy_SJ}), and
$\Ex \abs{h_\delta(t)} \le \Ex \abs{h_\delta(T)} = o_{\delta \downarrow 0}(1)$. 

Define now the stopping time
\begin{align}\label{eq:stoptime2}
\tau_L := \inf\left\{t > 0 : \int_0^t \bk{1 + \norm{R}_{H^1(\T)} 	
	+ \norm{S}_{H^1(\T)}}^2\,\d t' = L\right\}.
\end{align}
By Markov's inequality and the energy bound 
\eqref{eq:energy_SJ}, $\tau_L \xrightarrow{L\uparrow \infty} T$ a.s.: we have
\begin{align*}
\mathbb{P}(\{\tau_L < T\}) 
& = \mathbb{P}(\{\int_0^T \bk{1 + \norm{R}_{H^1_x} 	
	+ \norm{S}_{H^1_x}}^2\,\d t'  > L \})\\
&\le L^{-1} \Ex \int_0^T \bk{1 + \norm{R}_{H^1_x} 	
	+ \norm{S}_{H^1_x}}^2\,\d t' 
 \lesssim _\nu L^{-1},
\end{align*}
once more underscoring the necessity of $\nu > 0$. 
By Lemma \ref{thm:sto_gronwall_st}, 
\begin{align*}
\bk{\Ex \sup_{t' \in [0,T \wedge \tau_L]} \bk{\xi_\delta(t')}^{1/2}}^2
\lesssim e^{CLT}\Ex h_\delta(\tau_L\wedge T) .
\end{align*}
Therefore, we can take $\delta \downarrow 0$ first, 
so that the right side vanishes, and then take 
$L \uparrow \infty$ to conclude via the dominated convergence theorem.

\end{proof}

\subsection{Yamada--Watanabe principle}

Infinite dimensional versions of the Yamada--Watanabe 
principle have been derived in different settings. 
Most relevant for us is the following generalisation  
of \cite[Lemma 1.1]{GK1996} to the quasi-Polish case:
\begin{lem}[Quasi-Polish Gy\"ongy--Krylov lemma 
	{\cite[Theorem 2.10.3]{Breit:2018ab}}]\label{thm:qP_GKlemma}
Let $\mathcal{X}$ be a quasi-Polish space.  
Let $\{u_n\}_{n \in \N}$ be a sequence of 
random variables with laws tight in $\mathcal{X}$. 
Suppose that every subsequence $(u_n , u_m )_{n,m \in \N}$
admits a further subsequence such that its joint laws 
converges weakly* to a measure supported 
on the diagonal of $\mathcal{X}\times \mathcal{X}$. 
Then there exists an $\mathcal{X}$-valued 
random variable $u$ and a subsequence 
$u_{n_k} \to u$ in $\mathcal{X}$ in probability. 
\end{lem}

Paired with Theorem \ref{thm:pathwiseunique}, 
we can now show:
\begin{thm}
There exists a unique probabilistic strong solution to the 
viscous variational wave equation \eqref{eq:vvw}.
\end{thm}

\begin{proof}
The argument is standard. Let $(R_M, S_M)$ 
and $(R_N,S_N)$ be solutions to the $M$th 
and $N$th Galerkin approximations \eqref{eq:galerkin_lim}, 
respectively with initial data $(R^0_M, S^0_M)$ 
and $(R^0_N, S^0_N)$. Recall the definition of the 
product path space $\mathcal{Y}$ 
following \eqref{eq:pathspaces_defin}.
As in Section \ref{sec:SJthm} we can apply the 
Skorokhod theorem to a sequence of 
$\mathcal{Y}\times \mathcal{Y}$-valued 
sequence of random variables 
$$
(X_{M_\ell}, X_{N_\ell}), \qquad \text{where}\quad
X_N := 
(R_N,S_N, R_N,S_N, u_N, 
	R^0_N, S^0_N, W),
$$ 
with a tight sequence of laws $\mu_{M_\ell, N_\ell}$ on 
$\mathcal{Y}\times \mathcal{Y}$. 
The Skorokhod--Jakubowski theorem 
gives us a sequence of random variables 
\begin{align*} 
&\big(\tilde{X}^{(1)}_\ell, \tilde{X}^{(2)}_\ell \big)
\to
\big(\tilde{X}^{(1)},\tilde{X}^{(2)}\big), 
\quad \text{ in $\mathcal{Y}\times \mathcal{Y}$, $\tilde{\mathbb{P}}$-a.s.},
\end{align*}
where, (by the identifications of 
Lemma \ref{thm:xiR_zetaS}), for $i = 1,2$, 
\begin{align*}
\tilde{X}^{(i)}_\ell &:= (\tilde R_\ell^{(i)},\tilde S_\ell^{(i)}, 
	\tilde R_\ell^{(i)},\tilde S_\ell^{(i)}, \tilde{u}_\ell^{(i)}
	\tilde R^{0,(i)}_{\ell}, \tilde S^{0,(i)}_{\ell}, \tilde W_\ell), \\
\tilde{X}^{(i)} &:=
(\tilde R^{(i)},\tilde S^{(i)}, 
	\tilde R^{(i)},\tilde S^{(i)}, \tilde{u}^{(i)}
	\tilde R^{0,(i)}, \tilde S^{0,(i)}, \tilde W).
\end{align*}

Since the entire sequence $(R^0_N, S^0_N)$ converges 
$\mathbb{P}$-a.s.~in $(L^2(\T))^2$, we find that 
$$
\Big(\tilde R^{0,(1)}, \tilde S^{0,(1)}\Big) 
	= \Big(\tilde R^{0,(2)}, \tilde S^{0,(2)}\Big), 
	\quad \text{$\tilde{\mathbb{P}}$-a.s.}
$$
By Theorem \ref{thm:pathwiseunique}, we find that 
the measures $\mu_{M_\ell, N_\ell}$ converge to a 
measure taking vales on the diagonal of 
$\mathcal{Y}\times \mathcal{Y}$. Lemma 
\ref{thm:qP_GKlemma} then implies the existence 
of a solution in the original probability space.

\end{proof}

\section{Temporal continuity}\label{sec:tempcont}

In this section we show that $R$ 
and $S$ have continuous paths in 
$L^2(\T)$. Strong temporal continuity 
characterises the dissipativity arising from $\nu$, and 
the $(\omega, t, x)$-integrability of $\pd_x u$. 
This establishes the remaining 
claim of Theorem \ref{thm:main}, and 
proves the main theorem of this paper. 

\begin{prop}
Let $(R, S)$ be a pathwise  
solution to \eqref{eq:vvw} with initial condition 
$(R^0, S^0)$. then $R, S \in 
	L^{\overline{p}}(\Omega; C([0,T];L^2(\T)))$ for any $\overline{p} < 2 p_0$. 
\end{prop}

\begin{proof}
We use the strategy of \cite[Lemma D.1]{GHKP2022} 
and mollify \eqref{eq:vvw} with a mollifier $J_\delta$ 
to get $(R_\delta, S_\delta)  
	:= (R *J_\delta, S*J_\delta)$. 
We then show that the collection 
$\{(R_\delta, S_\delta)\}_{\delta > 0}$ is Cauchy in 
the metric space $
\big(L^{2}(\Omega; C([0,T];L^2(\T))) \big)^2$. 
For any small $\theta > 0$, the interpolation inequality 
\begin{align*}
\Ex \norm{f_\delta - f_\ep}_{C([0,T];L^2(\T))}^{2p_0 - \theta} 
&\lesssim 
\bk{\Ex \norm{f_\delta - f_\ep}_{C([0,T];L^2(\T))}^{2} }^{1/p}\\
&\qquad \times\bk{\Ex \norm{f_\delta 
	- f_\ep}_{L^\infty([0,T];L^2(\T))}^{2p_0}}^{1/q}, \\
\text{with }\quad p = 1+ \eta ,
& \quad q = \frac{1 + \eta}\eta, 
\quad \eta = \frac{2p_0 -2 - \theta}\theta,
\end{align*}
then shows that they are also Cauchy in 
$L^{2p_0 - \theta}(\Omega; C([0,T];L^2(\T)))$. 

Happily, we can dispense with commutators 
here as we are depending on the nearness 
of $\delta$ and $\ep$ to keep terms small, 
rather than on Gronwall's inequality and the 
difference between two (possibly) different 
solutions $R_1$ and $R_2$ (as in Theorem 
\ref{thm:pathwiseunique}). 
As a result, the special nonlinear structure of 
the equations matters less here.

For two small numbers $1 \gg \delta, \ep > 0$, 
we have 
\begin{align*}
&\frac12\bk{ \norm{R_\delta - R_\ep}_{L^2(\T)}^2(t)
	+  \norm{S_\delta - S_\ep}_{L^2(\T)}^2(t) }\\
&	= \frac12 \bk{ \norm{R^0_\delta - R^0_\ep}_{L^2(\T)}^2
	+  \norm{S^0_\delta - S^0_\ep}_{L^2(\T)}^2 }
	+ \int_0^t \sum_{j = 1}^5 I_j \,\d t'
	+ \int_0^t M\,\d W,
\end{align*}
where, in equivalent non-divergence form,
\begin{align*}
I_1 &:= \int_{\T} \bk{R_\delta - R_\ep}
	\bk{c(u) \,\pd_x R }*\bk{J_\delta - J_\ep}\,\d x\\
	&\qquad\,\, -   \int_{\T} \bk{S_\delta - S_\ep}
	\bk{c(u) \,\pd_x S }*\bk{J_\delta - J_\ep}\,\d x,\\
I_2 &:= \int_{\T} \bk{R_\delta - R_\ep - S_\delta + S_\ep}
	\bk{\tilde{c}(u) \bk{R^2 - S^2}}*\bk{J_\delta - J_\ep}\,\d x,\\
I_3 &:= -\nu\int_{\T}  \abs{\pd_x\bk{R_\delta - R_\ep}}^2 
	+ \abs{\pd_x\bk{S_\delta - S_\ep}}^2\,\d x, \\
I_4 &:=  \int_{\T} \bk{R_\delta - R_\ep}
	\bk{\sigma\, \pd_x \bk{\sigma \,\pd_x \bk{R + S}} }
		*\bk{J_\delta - J_\ep}\,\d x\\
	&\qquad\,\, +  \int_{\T} \bk{S_\delta - S_\ep}
	\bk{\sigma\, \pd_x \bk{\sigma \,\pd_x  \bk{R + S}} }
		*\bk{J_\delta - J_\ep}\,\d x,\\
I_5 &:=  \int_{\T}  
	\abs{\bk{\sigma \,\pd_x  \bk{R + S}} 
		*\bk{J_\delta - J_\ep}}^2\,\d x,\\
M &:= \int_{\T}  \bk{R_\delta - R_\ep}
	\bk{\sigma \,\pd_x  \bk{R + S}} 
		*\bk{J_\delta - J_\ep}\,\d x\\
	&\qquad\,\, + \int_{\T} \bk{S_\delta - S_\ep}
	\bk{\sigma \,\pd_x  \bk{R + S}}
		*\bk{J_\delta - J_\ep}\,\d x,
\end{align*}
and we have already performed the 
customary integration-by-parts in $I_3$ above. 

Estimating the terms one after another, 
we find
\begin{align*}
\abs{I_1} &\lesssim \norm{R_\delta - R_\ep}_{L^2_x}
		 \norm{\pd_x R}_{L^2_x}
+  \norm{S_\delta - S_\ep}_{L^2_x} \norm{\pd_x S}_{L^2_x},\\
\abs{I_2} &\lesssim \bk{\norm{R_\delta - R_\ep}_{L^\infty_x} 	
	+ \norm{S_\delta - S_\ep}_{L^\infty_x}} 
		\norm{R^2 - S^2}_{L^1_x}\\
	&\lesssim\bk{\norm{R_\delta - R_\ep}_{H^1_x} 	
	+ \norm{S_\delta - S_\ep}_{H^1_x}} 
	\bk{\norm{R}_{L^2}^2 + \norm{S}_{L^2_x}^2}.
\end{align*}
Using the Cauchy--Schwarz inequality and 
standard properties of mollification, 
with $R, S 
	\in L^{2p_0}(\Omega;L^2([0,T] ;H^1(\T)) \cap L^\infty([0,T];L^2(\T)))$ 
	(by \eqref{eq:energy_SJ}), 
\begin{align*}
\Ex \int_0^T \abs{I_1}\,\d t 
\lesssim \bk{\Ex \int_0^T \bk{\norm{R_\delta - R_\ep}_{L^2_x} 
	+  \norm{S_\delta - S_\ep}_{L^2_x}}^2\,\d t}^{1/2} 
	\xrightarrow{\delta, \ep \downarrow 0} 0,\\
\Ex \int_0^T \abs{I_2} 
\lesssim \bk{\Ex \int_0^T \bk{\norm{R_\delta - R_\ep}_{H^1_x} 
	+  \norm{S_\delta - S_\ep}_{H^1_x}}^2\,\d t}^{1/2} 
	\xrightarrow{\delta, \ep \downarrow 0} 0.
\end{align*}
Similarly  $\Ex \int_0^T \abs{I_5} \,\d t \to 0$ 
as $\delta, \ep \to 0$.
We integrate-by-parts in $I_4$ to get
\begin{align*}
-I_4 &= 
 \int_{\T} \bk{R_\delta - R_\ep}
	\bk{\sigma\,\pd_x \sigma \,\pd_x  \bk{R + S}}
		*\bk{J_\delta - J_\ep}\,\d x\\
	&\quad\,\, + \int_{\T} \bk{S_\delta - S_\ep}
	\bk{\sigma\, \pd_x \sigma \,\pd_x  \bk{R + S} }
		*\bk{J_\delta - J_\ep}\,\d x\\
&\quad\,\, +  \int_{\T} \pd_x\bk{R_\delta - R_\ep}
	\bk{\sigma^2\,\pd_x  \bk{R + S} }
		*\bk{J_\delta - J_\ep}\,\d x\\
	&\quad\,\, + \int_{\T} \pd_x\bk{S_\delta - S_\ep}
	\bk{\sigma^2 \,\pd_x  \bk{R + S} }
		*\bk{J_\delta - J_\ep}\,\d x.
\end{align*}
By the Cauchy--Schwarz inequality, we conclude 
as for $I_1$ and $I_2$ that 
$\Ex \int_0^T \abs{I_4} \,\d t \to 0$ as $\delta, \ep \to 0$.
Finally, by the BDG inequality and Young's inequality, 
\begin{align*}
&\Ex\sup_{t \in [0,T]}  \abs{\int_0^t M \,\d W}\\
& \lesssim \Ex \bk{\int_0^T \abs{M}^2\,\d t}^{1/2}\\
& \lesssim \Ex \bk{\int_0^T \norm{R_\delta - R_\ep}_{L^2_x}^2 
	\norm{\bk{\sigma \pd_x\bk{ R+ S}}*\bk{J_\delta - J_\ep}}_{L^2_x}^2\,\d t}^{1/2}\\
&\quad\,\,+  \Ex \bk{\int_0^T \norm{S_\delta - S_\ep}_{L^2_x}^2 
	\norm{\bk{\sigma \pd_x\bk{ R+ S}}*\bk{J_\delta - J_\ep}}_{L^2_x}^2\,\d t}^{1/2}\\
& \lesssim  \Ex \bk{\int_0^T 
	\norm{\bk{\sigma \pd_x\bk{ R+ S}}*\bk{J_\delta - J_\ep}}_{L^2_x}^2\,\d t 
	\times \sup_{t \in [0,T]}\norm{R_\delta - R_\ep}_{L^2_x}^2 }^{1/2}\\
&\quad\,\,+ \Ex \bk{\int_0^T
	\norm{\bk{\sigma \pd_x\bk{ R+ S}}*\bk{J_\delta - J_\ep}}_{L^2_x}^2\,\d t
	\times \sup_{t \in [0,T]}\norm{S_\delta - S_\ep}_{L^2_x}^2}^{1/2}\\
& \le C  \underbrace{\Ex \int_0^T\norm{ \bk{\sigma \pd_x\bk{ R+ S}}
	*\bk{J_\delta - J_\ep}}_{L^2_x}^2\,\d t}_{
		\xrightarrow{\delta, \ep \downarrow 0} 0}\\
&\quad\,\, + \frac14 \Ex \sup_{t \in [0,T]}\norm{R_\delta - R_\ep}_{L^2_x}^2 
+ \frac14\Ex\sup_{t \in [0,T]}\norm{S_\delta - S_\ep}_{L^2_x}^2.
\end{align*}
The final two terms can be absorbed into the left side.

Therefore gathering the estimates for 
$I_1$, $I_2$, $I_5$, $I_4$, and $M$, and 
using the non-positivity of $I_3$, we get
$$
\frac14\Ex \sup_{t \in [0,T]}
	\bk{ \norm{R_\delta - R_\ep}_{L^2(\T)}^2(t)
	+  \norm{S_\delta - S_\ep}_{L^2(\T)}^2(t) } 
	\xrightarrow{\delta,\ep \downarrow 0} 0,
$$
and $\{(R_\delta, S_\delta)\}_{\delta > 0}$ 
is Cauchy in $\bk{L^2(\Omega;C([0,T];L^2(\T)))}^2$. 
This proves the proposition since 
$R_\delta \to R$, $S_\delta \to S$ 
in $L^2_{\omega,t,x}$ {\em a priori}. 
\end{proof}

\section*{Acknowledgements}

The author is grateful to Kenneth Karlsen, 
Helge Holden, and Luca Galimberti for helpful 
conversations. The author also gratefully acknowledges 
support by the Research Council of Norway via the project \textit{INICE} (301538).

\subsection*{Competing interests}
The author has no relevant financial or non-financial interests to disclose.

\end{document}